%% file: main.tex
\documentclass[a4paper, 11pt]{amsproc}

\usepackage[
    a4paper,
    left=3.4cm,
    right=3.4cm,
    top=2.5cm,
    bottom=2.5cm
]{geometry}

\usepackage{algorithm}
\usepackage{algpseudocode}
\usepackage{defs}
\usepackage{orcidlink}
\usepackage{calc}
\usepackage{graphicx,wrapfig,lipsum}
\usepackage[dvipsnames]{xcolor}
\usepackage{tabularray}
\usepackage{empheq} % for boxing equations
\newtcolorbox{empheqboxed}{colback=white, 
    colframe=black,
    boxrule=0.25mm,
    width=\columnwidth,
    sharpish corners,
    top=-2mm, % default value 2mm
    left=2pt,
    bottom=5pt
}

\definecolor{metablue}{HTML}{0064E0}
\definecolor{metafg}{HTML}{1C2B33}
\definecolor{metabg}{HTML}{F1F4F7}
\definecolor{metabgdeep}{HTML}{D9EFFF}
\definecolor{metagreen}{HTML}{EAFFE8}
\definecolor{metagreen}{HTML}{FCFFEE}
\definecolor{metared}{HTML}{FFEAE8}
\definecolor{tealblue}{RGB}{235,248,248}
\definecolor{tealbluedark}{RGB}{185,215,215}
\definecolor{linkcolor}{RGB}{225,110,20}
\definecolor{urlcolor}{RGB}{0,140,100}

\newtheorem{theorem}{Theorem}[section]
\newtheorem{assumption}{Assumption}[section]
\newtheorem{remark}{Remark}[section]

\newtheorem{proposition}{Proposition}[section]
\newtheorem{lemma}{Lemma}[section]
\newtheorem{fact}{Fact}[section] 
\newtheorem{corollary}{Corollary}[section]

\newtheorem{example}{Example}[section]
\DeclareSymbolFont{extraup}{U}{zavm}{m}{n}
\DeclareMathSymbol{\varheart}{\mathalpha}{extraup}{86}
\DeclareMathSymbol{\vardiamond}{\mathalpha}{extraup}{87}
\DeclareMathSymbol{\varclub}{\mathalpha}{extraup}{84}
\DeclareMathSymbol{\vardspade}{\mathalpha}{extraup}{85}
\usepackage[framemethod=tikz,xcolor=true]{mdframed}
\newmdenv[backgroundcolor=tealblue, roundcorner=1pt, skipabove=4pt, linewidth=0pt, innertopmargin=4pt]{myframe}
\newmdenv[backgroundcolor=metabgdeep, roundcorner=0pt, skipabove=5pt, linewidth=0pt, innertopmargin=7pt]{myOCP}
\newmdenv[backgroundcolor=metared, roundcorner=10pt, skipabove=7pt, linewidth=0pt, innertopmargin=7pt]{myalgo}
\newmdenv[%
    leftmargin=0.5cm,
    backgroundcolor=yellow!10,%
    roundcorner=5pt,%
    tikzsetting={draw=red, line width=2.0pt}%
    ]{SpecialText}%

\colorlet{tealbluelight}{tealbluedark!55!white}
    \newmdenv[
  backgroundcolor=metabg,
  roundcorner=0pt,
  skipabove=4pt,
  linewidth=0pt,
  innertopmargin=7pt,
  tikzsetting={draw=black, line width=0pt},
  frametitlefont=\bfseries,
   frametitlebackgroundcolor=tealbluelight,
  frametitleaboveskip=0.5ex,
  frametitlebelowskip=0.5ex
]{myOCPboxTitle}

\newcounter{ocpproblem}
\crefname{ocpproblem}{Problem}{Problems}
\Crefname{ocpproblem}{Problem}{Problems}

\newenvironment{problembox}[1][]{%
  \refstepcounter{ocpproblem}%
  \begin{myOCPboxTitle}[frametitle={Problem~\theocpproblem#1}]%
}{%
  \end{myOCPboxTitle}%
}

\newcounter{ocpdefinition}[section]
\crefname{ocpdefinition}{Definition}{Definitions}
\Crefname{ocpdefinition}{Definition}{Definitions}

\renewcommand{\theocpdefinition}{\Roman{section}.\arabic{ocpdefinition}}

\newenvironment{definitionbox}[1][]{%
  \refstepcounter{ocpdefinition}%
  \begin{myOCPboxTitle}[frametitle={Definition~\theocpdefinition#1}]%
}{%
  \end{myOCPboxTitle}%
}
\newcommand{\black}[1]{{\color{black} #1}} 
 
\usepackage{cleveref}

\newcounter{proofstep}
\newcounter{proofstepH}                     % never resets; see note below
\renewcommand{\theproofstep}{Step \arabic{proofstep}}

\AtBeginEnvironment{proof}{\setcounter{proofstep}{0}}

\newcommand{\proofstep}[2]{%
  \stepcounter{proofstepH}%
  \refstepcounter{proofstep}%
  \textbf{\theproofstep}\label{#1} (#2).%
}

\title[FFT-CovSteer]{Continuous-Time Covariance Steering with Common Free-Final Time: Finite-Horizon Solutions and Infinite-Horizon Limits}

\author[A. Selim]{Akan Selim\,\orcidlink{0000-0002-6505-5796}}
\author[F. Liu]{Fengjiao Liu\,\orcidlink{0000-0002-3599-6755}}
\author[S. Ganguly]{Siddhartha Ganguly\,\orcidlink{0000-0003-2046-2061}}
\author[P. Tsiotras]{Panagiotis Tsiotras\,\orcidlink{0000-0001-7563-4129}}

\thanks{%
A. Selim, S. Ganguly and P. Tsiotras are with \faGroup\ Daniel Guggenheim School of Aerospace, \faUniversity\ Georgia Institute of Technology, \faMapMarker\  Atlanta, USA. F. Liu is with the Department of Electrical and Computer Engineering, FAMU-FSU College of Engineering, Tallahassee, FL 32310, USA}

\thanks{%
Contact Information: (AS) \faEnvelope\ \texttt{aselim8@gatech.edu}, (FL) \faHome\ \url{https://sites.google.com/aya.yale.edu/fengjiao-homepage/home}, \faEnvelope\ \texttt{fliu@eng.famu.fsu.edu}, (SG) \faHome\ \url{https://sites.google.com/view/siddhartha-ganguly}, \faEnvelope\ \texttt{sganguly41@gatech.edu}, (PT) \faHome\ \url{https://dcsl.gatech.edu/tsiotras.html}, \faEnvelope\ \texttt{tsiotras@gatech.edu}.
}

\begin{document}
% \keywords{Stochastic optimal control, covariance control, optimal transport.}

\maketitle

% \vspace{-10mm}
\begin{abstract}
This article studies the optimal \textit{common free-final time} problem for steering the state covariance of a continuous-time stochastic linear system between prescribed initial and terminal covariance matrices. 
We first establish a deterministic reformulation of the SDE-constrained free-final time stochastic optimal control problem (SOCP). For the ensuing SOCP, we provide necessary conditions for optimality and establish sufficient conditions for the optimal common final time to be finite.
Subsequently, we characterize the asymptotic behavior of the finite-horizon optimal solutions as the final time tends to infinity, the invariant sets and trajectories associated with this limiting regime, and derive the sensitivity of the Hamiltonian with respect to the common final time.
Finally, leveraging these sensitivities, we develop a trust-region line-search algorithm together with an infinite-horizon case detection method, and demonstrate its performance on three different problems: (a) an illustrative covariance-steering example, (b) a demonstration of spacecraft maneuver using covariance control, and (c) Gaussian mixture to Gaussian mixture steering. 
\end{abstract}


\tableofcontents

\input{sections/1-Introduction}
\input{sections/2-Problem_Setup}
\input{sections/3.1-Main_Results}
\input{sections/3.2-Main_Results}
\input{sections/3.3-Main_Results}
\input{sections/4.1-Numerics_snsvty}
\input{sections/4.2-Numerics_algo}
\input{sections/Conclusion}

\bibliographystyle{amsalpha}
\bibliography{refs}

\appendix

\input{sections/appendixA}
\input{sections/appendixB}
\input{sections/appendixC}

\end{document}

%% file: sections/1-Introduction.tex
\section{Introduction} \label{sect:1}

Stochastic optimal control (SOC) theory \cite{ref:Flem-Rish:DSOC:75} provides a rich set of mathematical tools for analyzing uncertain dynamical systems, and designing optimal decision-making strategies, perhaps in the presence of constraints.
A particularly {challenging subclass} of stochastic optimal control problems arises when the terminal time itself is not fixed in advance, but must be \emph{optimized together} with the control policy.
In full generality, this paradigm leads to \emph{stochastic optimal control and stopping} problems \cite{book_stop}, and when the stopped realizations must satisfy a distributional constraint, to the well-known \emph{Skorokhod Embedding Problem} \cite{SEP1,SEP-eng,ref:OptStop:AS:SG:AP:PT}.
Our setting here is much more focused; we study the practical setting of stochastic optimal control problems (SOCPs) with a common stopping time, wherein all realizations are stopped at the same time.

Terminal time is itself a critical performance metric that directly affects the overall effectiveness of the control policy~\cite{Betts,ref:VinBet-DynOpt}.
Nonetheless, such \emph{common free-final time} problems are mathematically intricate, numerically delicate while being highly relevant in practice and they arise whenever the duration of the task is itself a performance-critical design variable. Some representative examples include:
\begin{enumerate}[label={\textup{(\textsf{soc}-\alph*)}}, leftmargin=*, widest=b, align=left]
\item \label{soc:app:1} Aerospace guidance and control problems, such as fuel-optimal spacecraft maneuvers~\cite{behcet2}, \cite{ref:Time:Opt:Att:Man}, attitude slew maneuvers~\cite{slew}, and planetary landing or pinpoint landing missions~\cite{behcet1};
\item \label{soc:app:2} Robotic deployment and motion-planning tasks ~\cite{KondoWuKumarHow2026}, where a robot must reach a target region or complete a manipulation task while balancing time, energy, and safety; 
\item \label{soc:app:3} Multi-agent reconfiguration problems~\cite{ZhengLinTangHe2026}, where a swarm or networked team must reorganize from one spatial distribution to another while jointly optimizing mission completion time and control effort.
\end{enumerate}
In such settings, fixing the terminal time a priori can be conservative, inefficient, or perhaps incompatible with the desired requirements. 
Allowing it to vary, however, introduces a fundamental tradeoff: a longer horizon may reduce the required control effort, but it may also allow stochastic uncertainty to grow, dissipate, or approach an invariant regime, depending on the system dynamics.

\color{black}

\subsection{Related Literature} 

The problem of controlling the second moment of a stochastic system was first formulated in the late 1980s~\cite{intro1} with the aim of controlling the \textit{steady-state} covariance of a stochastic, continuous- or discrete-time~\cite{intro3} linear time-invariant (LTI) system over an infinite time horizon.
Its finite-horizon counterpart, commonly known as the Covariance Steering (CS) problem, was first introduced and completely solved, for the continuous-time case, in the seminal papers~\cite{cite3_1,cite3_2,cite3_3}, where also the direct connection of CS with the classical problems of
Schr\"{o}dinger bridges~\cite{sch,ref:SOC:SB:YC,ref:caluya2021wasserstein,ref:SB:lifted:path} and the optimal mass transport~\cite{villani_ot} was first established. 
Specifically, under the assumption that the input matrix is the same as the diffusion matrix, 
it was proved in \cite{cite3_1} that there exists a unique time-varying optimal control law solving the CS problem, which is, in fact, the solution to a Schr\"{o}dinger bridge problem between given Gaussian boundary distributions. 
The same reference also provided analytical solutions for optimal covariance controllers.
Additional results dealing with different noise models, such as mixed multiplicative and additive noises, non-Gaussian jumps, and other martingale noises with time-varying dynamics have been reported in \cite{cite12,cite16,cite17}.
The discrete-time version of CS has been addressed in~\cite{balci,cite4,cite15,GoldTsi:cdc17}.
For a complete solution of the CS problem in discrete time see~\cite{cite18}.

There has also been a surge of interest in expanding the covariance steering theory into various settings, such as partially-available information \cite{cite8,BakolasCovarianceControl2017,maity2025optimal}, chance-constrained problems \cite{cite4}, parametric disturbances \cite{cite9}, scenarios with unknown model parameters \cite{cite10}, unknown system model or noise bounds \cite{cite13,cite14} via data-driven techniques, bounded robust setting \cite{ref:parmar2026robust}, steering using several optimal transport distances such as Wassersetein \cite{ref:Wasscovsteer:AH:EBW}, Gromov-Wasserstein \cite{ref:CovSteer:X:DD:CDC:NH:SG:KK}, and Hilbert-Schmidt penalty \cite{ref:CovSteer:IX:AH:TS:Fixed},  and also settings adapted to nonlinear systems \cite{yu2023,bakolas}.
More recently, reachability and controllability questions of the covariance equation were investigated in \cite{cite19}. 
Along the same lines,~\cite{abdelgalil2025collective} investigated the related problem of the  controllability/reachability of the state transition matrix of a linear system, also fixing a flaw in the proof of
the controllability result in \cite{brockett2012notes}.
Recent applications of covariance steering include robust space mission design~\cite{benedikter2022convex,cite1,kumagai2025robust}, planetary landing~\cite{cite2}, stochastic dynamic games~\cite{cite5,chen2019}, stochastic multi-agent control~\cite{cite6,cite7,ChenSteeringDistri2018,rapakoulias2025steering}, and contact-rich robotic systems~\cite{shirai2023}, among many others.

However, given all these preceding developments, most covariance steering formulations \emph{prescribe the terminal time in advance},
which is clearly a restriction for applications highlighted in \ref{soc:app:1}--\ref{soc:app:3} and many others. Numerical line-search procedures can be used to tune the final time in specific examples~\cite{cite2}; however, such procedures do not, by themselves, explain the structure of the optimal final time, the associated transversality condition, or the possibility that the optimal value may be approached only in an infinite-horizon limit.

\subsection{Issues and Proposed Approach}
\label{subsec:intro:issues:cov:steer}

Since the time horizon enters the optimization problem, additional optimality conditions associated with variations of the terminal time need to be introduced. 
In deterministic optimal control, this role is played by the classical Hamiltonian transversality condition \cite{ref:VinBet-DynOpt,ref:liber}. In stochastic control, however, especially when the terminal requirement is imposed at the level of a distribution or a covariance matrix, the corresponding optimality condition is \emph{less transparent}.
Thus, the issue is not only how to solve the SOCP for a prescribed horizon, but also how to determine what condition should govern the choice of horizon and how this condition behaves in both \emph{finite}- and \emph{large-horizon} regimes.

This question is natural in uncertainty-aware control. Real-world guidance and control problems often involve stochastic disturbances and measurement noise. 
For safety-critical operations, stochastic optimal controllers are required to quantify and mitigate the inherent noise in the system.
In this context, \emph{covariance control/steering} provides a mathematically tractable framework for controlling the evolution of uncertainty in stochastic dynamical systems.

\subsection*{Structural Subtleties}

The \textit{common} free-final time covariance steering problem exhibits several phenomena that are not apparent in a purely fixed-horizon formulation. 
For example:
\begin{enumerate}[label={\textup{(\textsf{ft}-\alph*)}}, leftmargin=*, widest=b, align=left]
\item \label{feature:1} \textit{Short-horizon singularity:} for very small final times, nontrivial prescribed covariance boundary conditions may force arbitrarily large control effort.
\item \label{feature:2} \textit{Long-horizon turnpike behavior:} for sufficiently large final times, optimal covariance trajectories may approach invariant covariance structures and spend most of the horizon near them before departing to satisfy the terminal constraint.
\item \label{feature:3} \textit{Asymptotic final time behavior:} in some cases, the final time optimality condition is not attained at any finite horizon, but may, instead, be approached only in the limit, as the final time tends to infinity.
\item \label{feature:4} \textit{Unmatched control-noise effects:} when the control and noise channels are not assumed to coincide, the large-horizon covariance behavior can be singular and/or arbitrarily large, yielding control and termination strategies structurally different from the matched-channel case.
\end{enumerate}
The features \ref{feature:1}--\ref{feature:4} --- which we will highlight and technically elaborate throughout our exposition in the sequel --- show that free-final time problems are indeed \emph{delicate}.

\subsection{Key Questions and Technical Contributions}
\label{subsec:intro:que:con}

This paper studies the \textit{common} free-final time covariance steering problem for \emph{continuous-time stochastic LTI systems}. Our objective is to understand not only how to compute a candidate optimal final time, but also when such a finite time should be expected to exist, and what happens when the relevant optimality condition is attained only asymptotically.

The analysis is guided by three broad questions:
\begin{myframe}
\begin{enumerate}[label={\textup{(\textsf{Q}-\alph*)}}, leftmargin=*, widest=b, align=left]
\item \label{ques:1} What is the correct optimality condition for the common final time in covariance steering, and how can it be expressed through an appropriate Hamiltonian transversality condition?
\item \label{ques:2} When does the free-final time covariance steering problem admit a finite candidate optimal final time, and when can the corresponding optimality condition occur only in the infinite-horizon limit?
\item \label{ques:3} What structure do long-horizon optimal covariance trajectories approach, and can this structure be used to design numerical methods that distinguish finite-time solutions from asymptotic regimes?
\end{enumerate}
\end{myframe}
The questions \ref{ques:1}--\ref{ques:3} provide the conceptual roadmap for the paper and lead to the following contributions:
\begin{enumerate}[label={\textup{(\textsf{C}-\alph*)}}, leftmargin=*, widest=b, align=left]
\item 
We formulate the continuous-time covariance steering problem with a common free-final time and derive the associated Hamiltonian transversality condition through an equivalent deterministic covariance-control reformulation. 
This yields a necessary condition for candidate optimal terminal times in terms of Hamiltonian-zeros.\footnote{See Definition \ref{Def:HamZeros} ahead for a definition of Hamiltonian-zeros.}

\item  
We characterize the finite-versus-asymptotic nature of the Hamiltonian-zeros.
In particular, we establish sufficient conditions under which a finite candidate optimal final time exists, based on the small- and large-horizon behavior of the associated fixed-final time covariance steering problem.
We also identify the mechanism by which the Hamiltonian-zero condition can be approached only in the infinite-horizon limit.

\item 
We introduce the notions of \emph{Asymptotic Gateways} and \emph{Asymptotic Connections} to describe the invariant covariance structures governing long-horizon optimal trajectories.\footnote{See \S \ref{sect_5_subs_a} ahead for their definitions.}
These objects provide a turnpike-type interpretation of the large-horizon regime and motivate a trust-region-based line-search algorithm with an infinite-horizon detection mechanism for computing candidate optimal final times in general stochastic LTI covariance steering problems.
\end{enumerate}
To the best of our knowledge, this is \embf{the first systematic study} of continuous-time covariance steering with a common free-final time in the general stochastic LTI setting, including \emph{both finite-horizon optimality conditions and both small-time and infinite-horizon limiting behavior}.

\subsection{Organization}

In Section~\ref{Section_2}, we first present the notation used throughout the paper and introduce our main problem and its setting.
In Section~\ref{Section_3}, we derive the transversality condition for the optimal common final time using an equivalent deterministic formulation, since the classical Stochastic Minimum Principle (SMP)~\cite{ref2} does not provide necessary conditions for either stochastic optimal stopping or optimality with respect to a common free-final time.
In Section~\ref{Section_4}, we present sufficient conditions for the finiteness of the candidate optimal final time.
Next, in Section~\ref{Section_5}, we introduce the notions of Asymptotic Gateways and Asymptotic Connections 
that characterize the invariant sets and the behavior of the associated trajectories arising in the infinite-horizon limit, along which the Hamiltonian and its sensitivity with respect to the final time vanish asymptotically. 
In Section~\ref{Section_7}, we derive sensitivities of the Hamiltonian with respect to the final time and develop a trust-region-based line-search algorithm with an infinite-horizon detection mechanism to compute the optimal common final time for general stochastic LTI systems. 
Finally, in Section~\ref{Section_8}, we demonstrate our findings on three numerical examples and detailed proofs of our technical results are given in Appendix \ref{sec:appen:a}--\ref{Appendix_Thm}.

\subsection{Notation}\label{sect:notation}

We employ standard notations unless stated otherwise. 
We let \(\N \Let \aset{1,2,\ldots}\) denote the set of positive integers. For \(n \in \N\), we denote by \(\Ree^n\) the standard Euclidean space equipped with the usual norm \(\Ree^n\ni x\mapsto\norm{x}\). 
The nonnegative orthant of \(\Ree^n\) will be denoted by \(\Ree^n_{\ge 0}\).
For \(n\in\N\), \(I_n\) denotes the \(n\times n\) identity matrix. For \(n,m\in\N\) and \(M\in\Ree^{n\times m}\), \(\ker(M)\) and \(\rng(M)\) denote its kernel and range, respectively. 
The notation \(\norm{M}\) denotes the induced Euclidean norm, i.e., the maximum singular value, and \(\|M\|_{\rm F}\) denotes the Frobenius norm.
$\mathbb{S}^n$ denotes the set of symmetric matrices, and $\mathbb{S}_{+}^n$ and $\mathbb{S}_{++}^n$ denote the symmetric positive semi-definite and positive-definite matrices in $\mathbb{R}^{n \times n}$, respectively. 
The positive semidefinite square root of $M\in\mathbb S^n_+$ is denoted by $M^{\frac{1}{2}}$. We define the matrix symmetrization operator as $M \mapsto \mathcal{S}(M) \Let M + M\t$.
For $x\in \Ree^n$ and a set $\mathcal{C} \subset \Ree^n$, the distance of $x$ from $\mathcal C$ is denoted by ${\rm dist}(x,\mathcal C) \Let \inf_{y\in\mathcal C}\|x-y\|$. 
Whenever the argument is matrix-valued, the same notation is used with the corresponding matrix norm.
We denote by $\mathbb E[\;\cdot\;]$ the expectation operator and by $\mathcal N(\mu,\Sigma)$ the Gaussian distribution with mean $\mu\in \Ree^n$ and covariance $\Sigma \in \mathbb S^n_+$.
Finally, all random variables and stochastic processes arising in the sequel are understood to be defined on a rich enough complete filtered probability space satisfying the usual conditions.\footnote{We refer interested readers to \cite[Chapters 2, 3, and 5]{book1} and \cite[Chapters 2, 3, and 8]{ref:LeGall:StoCal:Book} for several definitions related to stochastic processes, SDEs, and probabilistic implications.}

%% file: sections/2-Problem_Setup.tex
\section{Problem Formulation} \label{Section_2}

\subsection{The Central Problem}\label{subsec:prob:central:SOCP}

We fix \(n,m,r \in \N\), and consider a controlled continuous-time stochastic differential equation governed by the following dynamics
\begin{align} 
\label{StochasticSystem}
\d x(t) = \bigl(Ax(t)+Bu(t)\bigr)\d t + \noise \,\d w(t) \quad\text{for all } t \ge 0,
\end{align}
with the following data: 
\begin{enumerate}[label={\textup{(\ref{StochasticSystem}-\alph*)}}, leftmargin=*, widest=b, align=left]
\item \label{prob:data:1} $x(\cdot)$ is an $\mathbb{R}^n$-valued state process, $u(\cdot)$ is an $\mathbb{R}^m$-valued control process;

\item \label{prob:data:2} $\noise \in\mathbb{R}^{n\times r}$ is the diffusion matrix, $w(\cdot)$ is a $\mathbb{R}^r$-valued standard Brownian motion and $A\in\mathbb{R}^{n\times n}$, $B\in\mathbb{R}^{n\times m}$ are constant matrices. We define \(D\Let \noise \noise \t\).
\end{enumerate}
The initial state is denoted by \(x_0\Let x(0)\), and is assumed to satisfy \(x_0\sim\mathcal{N}(0,\initcov )\) where $\initcov $ denotes the prescribed initial covariance. 
We assume that \(x_0\) is independent of \(w(\cdot)\), and let \(\{\filtration_t\}_{t\ge0}\) be the usual augmentation of the filtration generated by \(x_0\) and \(w(\cdot)\), that is, the $\sigma$-algebra \(\sigma(x_0,w(s):0\le s\le t)\). 
Here $\sigma(\cdot)$ denotes the $\sigma$-algebra generated by the random variables inside the parentheses.
We \emph{highlight that}, no uniform ellipticity condition is imposed, i.e., $D$ is allowed to be singular, such that the Brownian motion need not excite all directions of the state space.

The objective is to steer the state distribution of \eqref{StochasticSystem}, under a common \textit{free} final time $\tf>0$, to the prescribed terminal Gaussian distribution $x(\tf) \Let \xf\sim\mathcal{N}(0,\fincov )$, where $\fincov$ denotes the prescribed terminal covariance.
Since both the initial and terminal means are fixed at zero, the problem reduces to one of only steering the covariance.

The following standing assumptions are used throughout the paper.
\begin{assumption}[Standing assumptions]
\label{assump:standing}
We assume the following:
\begin{enumerate}[label={\textup{(\ref{assump:standing}-\alph*)}}, leftmargin=*, widest=b, align=left]
    \item \label{assump:standing:1} The pair $(A,B)$ is controllable.

    \item \label{assump:standing:2} The prescribed endpoint covariances satisfy $\initcov ,\fincov \in\mathbb{S}_{++}^n$ and $\initcov \neq \fincov$.
\end{enumerate}
\end{assumption}

A control input $\uu$ is said to be \emph{admissible} on $[0, +\infty)$ if it is nonanticipative and has finite expected energy on $[0, +\infty)$. In particular, $u(\cdot)$ may depend on $t$ and, possibly, on the state history $\aset[]{x(s)\suchthat 0\le s\le t}$, but not on future values of the state. 
More precisely, define the set of all such \emph{admissible control inputs} by
\begin{align}
\mathcal U \Let \left\{ \uu\,\middle|\,
\begin{array}{l}
\uu \text{ is progressively measurable with respect to } \{\filtration_t\}_{t\ge0}\\ \text{and }
\mathbb E\left[\int_0^{+\infty} u(t)\t u(t)\,\d t\right]<+\infty
\end{array} \right\}.
\end{align}
For every admissible control \(u(\cdot) \in \mathcal{U}\), under the data \ref{prob:data:1}--\ref{prob:data:2}, the controlled SDE \eqref{StochasticSystem} admits a \emph{unique strong solution} adapted to the usual augmentation of the filtration generated by \(x_0\) and \(w(\cdot)\) \cite[Chapter 1, Section 6.4]{ref2}.

For a given admissible control \(u(\cdot)\in\mathcal{U}\) and a given \emph{common free-final} time \(\tf>0\),we  consider the performance index
\begin{align} \label{eq:cost}
\bigl(\tf, u(\cdot)\bigr) \mapsto J(\tf, u(\cdot)) \Let  \E\left\{ \half \int_0^{\tf}  \big[\varpi +x(t)\t Q x(t) + u(t)\t R u(t)\big] \,\d t \right\},
\end{align}
where $Q\succcurlyeq 0$, $R\succ 0$, and $\varpi\ge 0$. Here is our central object of study.
\begin{problembox}[ (Covariance Steering Problem)]\label{prob:Cov:Steer}
Consider the SDE \eqref{StochasticSystem} along with its data \ref{prob:data:1}--\ref{prob:data:2}. Find $\tf > 0$ and $u(\cdot)\in \mathcal{U}$ that solve the SOCP
\begin{equation}
\begin{aligned}
& \inf_{\tf>0,\;u(\cdot)\in\mathcal{U}} 	&&  J\bigl(\tf,u(\cdot)\bigr)  \\
& \hspace{6mm} \sbjto		&&   \begin{cases}
\text{SDE}\,\,\eqref{StochasticSystem} \text{ with }\ref{prob:data:1}\text{--}\ref{prob:data:2},\\ 
x_0\sim\mathcal{N}(0,\initcov ) \text{ and }\xf\sim\mathcal{N}(0,\fincov ) \\
\text{with }\initcov ,\fincov \in \mathbb{S}_{++}^n,~ \initcov \ne\fincov.
\end{cases}
\end{aligned}\nn
\end{equation}
\end{problembox}
Notice that in \Cref{prob:Cov:Steer}, we optimize \eqref{eq:cost} over not only the admissible controls $u(\cdot) \in \mathcal{U}$, but \emph{also} the free-final time $\tf$. 

%% file: sections/3.1-Main_Results.tex
\section{Deterministic Reformulation} \label{Section_3}
In this section, we will provide the necessary conditions for \Cref{prob:Cov:Steer}. This will be achieved by reformulating \Cref{prob:Cov:Steer} as a deterministic OCP, for which we can invoke the standard Pontryagin Maximum Principle (PMP) \cite{ref:VinBet-DynOpt,ref:liber}.
\begin{myOCP}
\begin{proposition}\label{Prop:IV1}
Consider \Cref{prob:Cov:Steer} along with its data \ref{prob:data:1}--\ref{prob:data:2} and suppose that \Cref{prob:Cov:Steer} admits an optimal solution $\bigl(u\as(\cdot),\tf \as \bigr) \in \mathcal{U} \times (0,+\infty)$. Then, \(u\as(\cdot)\) admits a linear state-feedback representation
\begin{align} \label{opt_controller}
  [0,\tf] \ni t \mapsto  u\as(t) = K\as(t) x(t),
\end{align}
where $K\as:[0,\tf]\to\Ree^{m\times n}$ is the optimal time-varying feedback gain. 
\end{proposition}
\end{myOCP}
A detailed proof of Proposition~\ref{Prop:IV1} is given in Appendix~\ref{sec:appen:a}. Proposition~\ref{Prop:IV1} identifies the structural form of the optimal controller and furnishes the feedback gain explicitly in terms of a (Riccati matrix-valued) trajectory \(t \mapsto S(t)\). 
However, this characterization alone \emph{does not determine} the optimal free-final time.  To address the free-final time component of the problem, we next proceed to an equivalent deterministic covariance-control formulation and derive the corresponding transversality condition for the associated Hamiltonian.

With the controller parameterization \eqref{opt_controller}, 
the problem becomes one of finding a suitable \emph{time-varying gain matrix} $K:[0,+\infty)\to\Ree^{m\times n}$.
We say that \(K(\cdot)\) is \emph{admissible} if it is locally piecewise continuous on \([0,+\infty)\). 
Specifically, we define the set of admissible gains
\begin{align*}
\mathcal{K}\Let \mathrm{PC}_{\rm loc}\bigl([0,+\infty),\Ree^{m\times n}\bigr),    
\end{align*}
where \(\mathrm{PC}_{\rm loc}\bigl([0,+\infty),\Ree^{m\times n}\bigr)\) denotes the set of \(\Ree^{m\times n}\)-valued functions that are piecewise continuous on every compact subinterval of \([0,+\infty)\).  

Observe that under the control law \eqref{opt_controller}, the stochastic dynamics \eqref{StochasticSystem} induce deterministic covariance dynamics, and the state distribution remains Gaussian with \emph{zero mean} for all $t\in[0,\tf]$, for a given common final time $\tf$. Moreover, the performance index \eqref{eq:cost} reduces to a deterministic objective functional depending only on the covariance dynamics of the state. Since the law \eqref{opt_controller} is optimal for \Cref{prob:Cov:Steer}, we have the following \emph{deterministic} OCP.

\medskip

\begin{problembox}[ (Deterministic Reformulation)]\label{prob:Det:Cov:Steer}
Consider the SDE \eqref{StochasticSystem} along with its data \ref{prob:data:1}--\ref{prob:data:2}. In view of Proposition \ref{Prop:IV1}, solving \Cref{prob:Cov:Steer} is \emph{equivalent} to solving the OCP
\begin{subequations}
\begin{align}
\inf_{\tf>0,\;K\in \mathcal{K}} & \hspace{3mm}J^{\rm P2}(\tf,K) \Let  \half \int_0^{\tf} \big(\varpi +\trace\big(Q\Sigma(t) \big) \nonumber \\
& + \trace\big( K(t)\t R K(t)\Sigma(t)\big)\big)\,\d t,\label{det_opt_problem}
\end{align}
\hspace{27mm}\(\sbjto\)
\begin{align}
\hspace{27mm}&\dot{\Sigma}(t)=\left(A+BK(t)\right)\Sigma(t)+\Sigma(t)\left(A+BK(t)\right)\t+D,\label{det_opt_problem_cov_prop}\\
\hspace{27mm}&\Sigma(0)=\initcov ,~~ \Sigma(\tf)=\fincov  \label{mean_bdrS}. 
\end{align}
\end{subequations}
\end{problembox}

\Cref{prob:Det:Cov:Steer} is \emph{equivalent} to \Cref{prob:Cov:Steer} in the sense that it induces the \emph{same optimal value} of the performance index \eqref{eq:cost} and the same optimal feedback controller. 
\begin{remark}
Since \Cref{prob:Det:Cov:Steer} is deterministic, we can employ standard tools from deterministic optimal control theory \cite[Chapters 7-13]{ref:VinBet-DynOpt} for its analysis. 
Moreover, minimizing the corresponding fixed-horizon value over $\tf$ preserves the equivalence for the common free-final time problem, because the equivalence holds for \emph{every fixed $\tf>0$.}  
\end{remark}

We start our analysis by providing the Hamiltonian and the necessary conditions for optimality for \Cref{prob:Det:Cov:Steer}. 

\begin{myOCP}
\begin{theorem}\label{Theorem1}
    Consider \Cref{prob:Det:Cov:Steer} with its associated data and let $\tfc \in (0,+\infty)$ be a candidate optimal final time. Let $\bigl(\Sigma(\cdot),\Pi(\cdot)\bigr)$ denote the corresponding optimal state and co-state trajectories, and let $(\mathbb{S}_{++}^n \times \mathbb{S}^n)\ni (\Sigma,\Pi)\mapsto \hamil_{\tfc}(\Sigma,\Pi)\in \Ree$ defined by
 \begin{align} \label{final_ham_dyn}
     \hamil_{\tfc}(\Sigma,\Pi)  \Let \half \varpi +\trace\Big( \half Q\Sigma +\Pi  A\Sigma \Big) - \trace \Big( \half\Pi B R^{-1} B\t\Pi \Sigma  - \half\Pi D\Big),
    \end{align}
    be the associated Hamiltonian.
    Then, along the optimal trajectory $[0,\tfc]\ni t\mapsto(\Sigma(t),\Pi(t))$, the Hamiltonian satisfies the transversality condition
    \begin{align} \label{total_hamiltonian}
    \hspace{-1 em} \hamil_{\tfc} (\Sigma(t),\Pi(t)) = 0 \quad \text{for all }t\in[0,\tfc],
    \end{align}
    where the symmetric matrix-valued co-state trajectory $t\mapsto \Pi(t)\in\mathbb{S}^n$ satisfies
    \begin{equation}  \label{final_ricc_dyn}
        \dot{\Pi}(t)=-Q-A\t\Pi(t)-\Pi(t)A  +\Pi(t)BR^{-1}B\t\Pi(t) ,
    \end{equation}  
    with free endpoints $\Pi(0)$ and $\Pi(\tfc)$.
\end{theorem}
\end{myOCP}
A proof of Theorem \ref{Theorem1} is included in Appendix \ref{sec:appen:a}. Next, we recall the finite-horizon well-posedness result from~\cite{cite17} that will be used throughout the paper.

\begin{fact}[Fixed-horizon well-posedness]
\label{rem:fixed_horizon_wellposedness}
For every fixed finite final time $\tf\in(0,+\infty)$, the co-state/covariance two-point boundary value problem in \eqref{final_ricc_dyn}--\eqref{final_cov_dyn} associated with \Cref{prob:Det:Cov:Steer} admits a \emph{unique solution}.
In particular, from \cite[Lemma~8]{cite17}, it follows that, the map from $\Pi(0)\in \mathbb{S}^n$ to $\fincov  \in \mathbb S^n_{++}$ is a homeomorphism, while \cite[Theorem~3]{cite17} established the uniqueness of the solution of the coupled equations. 
The case $\tf\to + \infty$ will be discussed in Section~\ref{Section_5}.
\end{fact}

\subsection{Infinite Horizon Case} \label{Section_InfHor}

A special subclass of solutions to \Cref{prob:Det:Cov:Steer} consists of infinite-horizon trajectories. 
We focus on the class of problems for which the limiting trajectories admit a well-defined pointwise limit as $t\to+\infty$. 
This motivates the following assumption. 
Below,  \(\{\Sigma_{\tf}(\cdot)\}_{\tf>0}\) denotes the family of finite-horizon optimal covariance trajectories indexed by the prescribed terminal time \(\tf\).
That is, for each fixed final time \(\tf>0\),  \([0,\tf]\ni t\mapsto \Sigma_{\tf}(t)\) 
denotes the optimal covariance trajectory of the fixed-final time version of \Cref{prob:Det:Cov:Steer} on the interval \([0,\tf]\), whenever such an optimizer exists.

\begin{assumption}[Asymptotic point convergence]\label{assump:point_limit}
Consider the family of optimal covariance trajectories $\{\Sigma_{\tf}(\cdot)\}_{\tf>0}$ that solve the fixed-final time version of \Cref{prob:Det:Cov:Steer}. We assume that the limiting trajectory $t\mapsto \Sigma_{\star}(t)$ of $\{\Sigma_{\tf}(\cdot)\}$ obtained as $\tf\to+\infty$ has a bounded positive-definite limit $\Sigma_\infty\succ 0$, that is,
\begin{align*}
    \lim_{t\to+\infty}\Sigma_{\star}(t)=\Sigma_\infty.
\end{align*}
\end{assumption}

\begin{remark}\label{rem:on:Assump:III.1}
Assumption \ref{assump:point_limit} will be employed only in the analysis of sufficiently large horizons in Section~\ref{Section_5} to justify the asymptotic analysis, and to ensure that the corresponding infinite-horizon limiting trajectories admit a well-defined covariance limit. 
In the turnpike literature (see, e.g., \cite{turnpike_survey}  for a recent survey), more complicated limiting behaviors, including periodic and even chaotic invariant motions, have been reported. 
For example, oscillatory infinite-horizon behavior in Riccati dynamics is discussed in~\cite{riccati_inf}.
Establishing sufficient conditions under which the Assumption~\ref{assump:point_limit} holds for \Cref{prob:Det:Cov:Steer} is beyond the scope of this paper. 
\end{remark}

%% file: sections/3.2-Main_Results.tex
\section{Hamiltonian-Zeros: Background and Definitions}\label{Section_4}

In this section, for the readers' convenience, we provide some background on the notion of \emph{optimal free-final time}. To set up the general framework, we first define the value function for a generic \emph{fixed} final time OCP with given endpoint constraints.
We then introduce the notion of \embf{Hamiltonian-zeros}, which characterizes optimality with respect to the final time. Finally, we investigate the finiteness of the Hamiltonian-zeros for \Cref{prob:Det:Cov:Steer}.

\medskip

\begin{definitionbox}[ (Optimal Final Time)] 
Consider the deterministic controlled system 
\begin{align*}
\dot{x}(t) = f(x(t),u(t))\quad \text{for } t\ge 0,   
\end{align*}
with the endpoint constraints $x(0)\Let x_0$ and $x(\tf) \Let \xf$; $t \mapsto x(t)\in\mathbb{R}^n$ and $t \mapsto u(t)\in\mathbb{R}^m$ are the state and the control trajectories, respectively. 
Let hte vector field $f: \mathbb{R}^n \times \mathbb{R}^m \to \mathbb{R}^n$ be Lipschitz, and let $\tf>0$ be a given final time, and
define the fixed-final time value function
\begin{align} \label{value_obj}
   V(\tf)\Let\inf_{u(\cdot) \in \mathcal{U}} {J}(\tf, u(\cdot)).
\end{align}
Then, the optimal free-final time associated with the cost $J(\tf, u(\cdot))$ is defined by $\tf^* \in \arginf\limits_{\tf > 0} V(\tf)$.
\end{definitionbox}

For the OCP \eqref{value_obj}, we recall the standard fact \cite[Chapter~V, Section~5.1]{ref3}:

\begin{fact} \label{ham_val_rel_lemma}
Consider the OCP \eqref{value_obj} along with its data. Suppose that \(t \mapsto V(t)\) is differentiable at \(\tf\) and let \(t \mapsto \bigl(x_{\tf}^*(t),u_{\tf}^*(t),p_{\tf}^*(t)\bigr)\) be a normal optimal extremal associated with the OCP \eqref{value_obj}. Let \(\hamil_{\tf}\) be the corresponding Hamiltonian. 
Then, \(V(\cdot)\) satisfies
\begin{equation} \label{dV_H_rel}
\frac{\d V}{\d \tf}(\tf)=\hamil_{\tf}\bigl(\xf, u_{\tf}^*(\tf), p_{\tf}^*(\tf)\bigr).
\end{equation}
\end{fact}

Adapting Fact~\ref{ham_val_rel_lemma} to our setting and leveraging the necessary optimality condition in Theorem~\ref{Theorem1}, we obtain the immediate corollary.

\begin{myOCP}
\begin{corollary} \label{cor_dv_H_0}
Consider Problem \Cref{prob:Det:Cov:Steer}. For each fixed \(\tf>0\), let $\ V(\tf)\Let\inf_{K \in \mathcal{K}} J^{\rm P2}(\tf, K)$ denote the value function of the fixed-final time version of \Cref{prob:Det:Cov:Steer}, where $J^{\rm P2}(\cdot)$ is as in \eqref{det_opt_problem}. Assume that $V(\cdot)$ is differentiable at \(\tf\). 
Suppose that an optimal final time $\tf^* < + \infty$ exists for \Cref{prob:Det:Cov:Steer}, and let $\bigl(\Sigma_{\tf^*}(\cdot),\Pi_{\tf^*}(\cdot)\bigr)$ be the associated optimal state and co-state trajectory.
Then,
\begin{align} \label{eq:dv-H}
    \frac{\d V}{\d \tf}(\tf^*)=\hamil_{\tf^*} \bigl(\Sigma_{\tf^*}(\tf^*), \Pi_{\tf^*}(\tf^*)\bigr)=0,
\end{align}
and $\tf^* = \arg\inf\limits_{\tf > 0} V(\tf)$.
\end{corollary}
\end{myOCP}

Throughout this paper, we employ $\tf^*$ to denote the global optimal final time solution, and $\tfc$  to denote a candidate optimal solution that satisfies (possibly asymptotically)  the transversality condition \eqref{total_hamiltonian}.
With a slight abuse of notation, we will use the shorthand 
$\hamil_{\tf}(\cdot)$ to denote the optimal Hamiltonian corresponding to the \textit{fixed} final time $\tf$ problem. We now introduce one of the central objects of this article. 
\begin{definitionbox}[ (Hamiltonian-zeros)] \label{Def:HamZeros}
Consider \Cref{prob:Det:Cov:Steer} and suppose that Assumption~\ref{assump:standing} holds.
Then,
   \begin{enumerate}[label={\textup{(\ref{Def:HamZeros}-\alph*)}}, leftmargin=*, widest=b, align=left]
        \item \label{Def:HamZeros:a} A candidate final time $\tfc \in (0,+\infty)$ is called a \textbf{Hamiltonian-zero} if the transversality condition \eqref{total_hamiltonian} is satisfied.
             
        \item \label{Def:HamZeros:b} A candidate final time $\tfc=+\infty$ is called an \textbf{asymptotic Hamiltonian-zero} if $\lim\limits_{\tf \to +\infty}\hamil_{\tf}(\cdot) = 0$.
        \end{enumerate}
\end{definitionbox}

In the case of an asymptotic Hamiltonian-zero, the vanishing of the Hamiltonian is understood as a limiting property of the \emph{family} of finite-horizon optimal solutions to the fixed final time version of \Cref{prob:Det:Cov:Steer} as $\tf \to +\infty$. More importantly, the existence of an asymptotic Hamiltonian-zero does not ensure the existence of an optimal controller for the corresponding infinite-horizon problem that satisfies the prescribed terminal covariance boundary condition.

\begin{remark}[Relation with Infinite-horizon Optimality Notions]
\label{rem:inf_notions}
The distinction made in the preceding paragraph, between treating $\tfc=+\infty$ as a limiting final-time candidate and requiring the existence of an infinite-horizon optimal controller is standard in the infinite-horizon optimal control literature~\cite{Carlson1991}.
Specifically, when the cost functional is defined through an improper integral, or when the terminal conditions are imposed only asymptotically, classical optimality may fail to be attained or may not be the most appropriate comparison criterion.
In such cases, notions such as \emph{overtaking} and \emph{weak overtaking} optimality are commonly employed~\cite{Carlson1991}.
For such problems, the related boundary conditions at infinity have been studied for asymptotic terminal constraints and state-constrained problems~\cite{Khlopin2023,Basco2018}. 
We do not formulate \Cref{prob:Det:Cov:Steer} in terms of these criteria. 
Rather, we mention these infinite-horizon optimality notions  only to clarify why an asymptotic Hamiltonian-zero should be viewed as a limiting finite-horizon object rather than necessarily as the result of an attained infinite-horizon optimizer.
\end{remark}

We are now ready to state a finiteness criterion for candidate optimal final times of \Cref{prob:Det:Cov:Steer}. 
The theorem is stated here for clarity; its proof relies on the small-time and large-time Hamiltonian asymptotics developed in the subsequent sections. 
The intuition behind this result stems from Corollary~\ref{cor_dv_H_0}, which states that any finite local minimizer of the fixed-final time value function must satisfy the Hamiltonian-zero condition \eqref{eq:dv-H}. 
Consequently, if the mapping $\tf \mapsto \hamil_{\tf}$ is continuous, a Hamiltonian-zero exists if one can show that $\hamil_{\tf}<0$ for all sufficiently small $\tf>0$, while $\hamil_{\tf}>0$ for all sufficiently large $\tf$. 
The following theorem gives simple sufficient conditions under which $\hamil_{\tf}(\cdot)$ changes sign.

\begin{myOCP}
 \begin{theorem}[Sufficiency for Finiteness of Free-Final Time]
\label{thm:p1d-main}
Consider \Cref{prob:Det:Cov:Steer} along with its data and notation, and let Assumptions~\ref{assump:standing} and ~\ref{assump:point_limit} hold.
Assume that one of the following conditions holds:
\begin{enumerate}[label={\textup{(\ref{thm:p1d-main}-\alph*)}}, leftmargin=*, widest=b, align=left]
    \item \label{thrm:suff:finite:1} $\varpi>0$,
    \item \label{thrm:suff:finite:2} $Q \succeq 0$ and $Q \neq 0$,
    \item \label{thrm:suff:finite:3} The Lyapunov equation $A \Sigma_{\rm s} + \Sigma_{\rm s} A\t + D=0$ admits no positive-definite solution $\Sigma_{\rm s}$.
\end{enumerate}
Then, \Cref{prob:Det:Cov:Steer} admits at least one finite candidate free-final time minimizer $\tfc$. 
\end{theorem}   
\end{myOCP}
The proof of the theorem requires additional technical results developed in Section~\ref{Section_5} and Appendix~\ref{sec:appen:b}, and thus is given in Appendix~\ref{Appendix_Thm}.

Next, we introduce the associated concepts of \embf{Asymptotic Gateways} (AGs) and \embf{Asymptotic Connections} (ACs) to characterize the behavior of candidate Hamiltonian-zero solutions $\tfc$, 
and study the asymptotic behavior of the Hamiltonian for \Cref{prob:Det:Cov:Steer}. The supporting boundedness, continuity, and compactness properties of the solutions of the fixed-time version of \Cref{prob:Det:Cov:Steer} are technical and are given in Appendix~\ref{sec:appen:b}.

%% file: sections/3.3-Main_Results.tex
\section{Hamiltonian-Zeros of Covariance Steering} \label{Section_5}
In this section, we study the large-horizon behavior of Hamiltonian-zeros for \Cref{prob:Det:Cov:Steer}. 
The central question is whether a candidate Hamiltonian-zero can escape to the asymptotic regime \(\tf=+\infty\), or whether it must occur at a finite final time.
We show that this dichotomy is governed by a stationary covariance set associated with the uncontrolled covariance dynamics. 
Since this section represents the \emph{technical core} of the article, we first provide an overview and structure of the results in the sequel. 
\begin{enumerate}[label={\textup{(\ref{Section_5}-\alph*)}}, leftmargin=*, widest=b, align=left]
\item 
We first study the covariance--Riccati system for the fixed-final time problem and show that finite-horizon covariance trajectories remain in \(\mathbb{S}^n_{++}\). The accompanying continuity, boundedness, and compactness properties of the co-state/covariance trajectories, needed to extract limiting arcs as \(\tf\to+\infty\), are technical and are deferred to Appendix~\ref{sec:appen:b} (Lemma~\ref{lem:reg:init_cost} and Proposition~\ref{cor:cov_costate_compactness}).
\item 
We introduce the \emph{Asymptotic Gateway} set \(\asmgate\), which consists of stationary covariance states that incur zero state-running cost.
The main point of this construction is that \(\asmgate\) identifies the stationary covariance states near which a long-horizon trajectory can remain without accumulating additional covariance-running cost when \(\varpi=0\).

\item 
We define \emph{Asymptotic Connections} as limiting covariance trajectories that approach \(\asmgate\). 
These connections describe the \emph{turnpike-type structure} of long-horizon optimal trajectories: \emph{they approach the set  \(\asmgate\), remain near it, and then depart to meet the terminal covariance constraint.}
The primary result here is that Asymptotic Connections give the precise limiting object for finite-horizon optimal trajectories as the final time tends to infinity, without requiring the existence of an attained infinite-horizon optimizer satisfying the terminal covariance constraint.

\item  
We then prove that \(\asmgate \neq\varnothing\), together with \(\varpi=0\), produces an asymptotic Hamiltonian-zero, while \(\asmgate=\varnothing\) rules out such a possibility.
The key result of this part is the \emph{large-horizon dichotomy}: the nonemptiness of \(\asmgate\) is exactly the mechanism by which the Hamiltonian can vanish asymptotically, whereas \(\asmgate=\varnothing\) forces the large-horizon Hamiltonian to remain strictly positive.

\item 
Combining the small-time limit \(\hamil_{\tf}\to-\infty\) as \(\tf\to 0^+\) with the large-horizon limit of \(\hamil_{\tf}\)
shows that \(\tf=+\infty\) is an asymptotic Hamiltonian-zero precisely when \(\varpi=0\) and when there exists \(\Sigma_s\in\mathbb{S}^n_{++}\) such that
\[
A\Sigma_s+\Sigma_sA\t+D=0\quad\text{and}\quad \trace(Q\Sigma_s)=0.
\]
In other words, the asymptotic case can arise only when the uncontrolled covariance dynamics admit a positive-definite stationary covariance on which the running covariance cost vanishes. If this condition fails, then the large-horizon Hamiltonian is positive, while the small-time Hamiltonian is negative, by continuity, a finite Hamiltonian-zero exists.
\end{enumerate}

\subsection{Covariance--Riccati System} \label{sect5_a}

We first study the covariance--Riccati system associated with the fixed-final time version of \Cref{prob:Det:Cov:Steer}. 
To this end, introduce the Hamiltonian matrix
\begin{align} \label{ham_matr}
\Ree^{2n\times 2n} \ni \matHamil \Let \begin{bmatrix}
A & - B R^{-1} B\t \\
- Q & -A\t
\end{bmatrix},
\end{align}
associated with the Riccati--covariance dynamics~\eqref{final_ricc_dyn},~\eqref{final_cov_dyn},
and let the associated state-transition matrix
\begin{align*}
  [0,\tf]^2 \ni (t,\tau)\mapsto  \bar \Phi(t,\tau) \Let \exp\bigl(\mathcal{H}(t-\tau)\bigr) \in \Ree^{2n\times 2n},
\end{align*}
be partitioned as
\begin{align} \label{STM_matr}
\hspace{-3mm} \bar{\Phi}(t,\tau)=\begin{bmatrix}
\Phi_{11}(t,\tau) & \Phi_{12}(t,\tau) \\
\Phi_{21}(t,\tau) & \Phi_{22}(t,\tau)
\end{bmatrix}\,\,\text{with}\,\, \bar{\Phi}(\tau,\tau) = I_{2n},
\end{align}
where ${I}_{2n}\in\mathbb{R}^{2n \times 2n}$ denotes the identity matrix. Then, the covariance trajectory $[0,\tf] \ni t \mapsto \Sigma(t)$ with fixed final time $\tf$ can be expressed as~\cite{cite17}
\begin{align} \label{cov_analytic_stm}
   t \mapsto \Sigma(t) = \bar{\Phi}_{\bar A}(t,0)\initcov  \bar{\Phi}_{\bar A}(t,0)\t 
   + \int_0^t \bar{\Phi}_{\bar A}(t,s) D \bar{\Phi}_{\bar A}(t,s) \t \, \d s,
\end{align}
where $\bar A (t) \Let A -BR^{-1}B\t \Pi (t)$ with $\Pi (t)$ as in \eqref{final_ricc_dyn}, 
and for all \(s,t \in[0,\tf]\),
\begin{align} \label{STM_Cov}
    \bar{\Phi}_{\bar A}(t,s) \Let \Phi_{11}(t,s)+\Phi_{12}(t,s)\Pi(s),
\end{align}
is the state transition matrix of $\bar A(t)$~\cite[Lemma~7]{cite17}.

Covariance trajectories evolve within $\mathbb{S}_{++}^n$ over every finite horizon, provided the boundary covariances satisfy $\initcov \succ0$ and $\fincov \succ0$.
Indeed, since $\bar \Phi_{\bar A}(t,0)$ is nonsingular for all finite $t\in[0,\tf]$ and $D \succcurlyeq 0$, it follows that for any admissible $K\in\mathcal{K}$, the first term from \eqref{cov_analytic_stm} is positive-definite, and the integral term is positive semidefinite. 
Hence, $\Sigma(t)\succ0$ for all $t\in[0,\tf]$. In particular, the optimal covariance trajectory cannot touch the boundary of $\mathbb{S}_+^n$ in finite time. 
This does not, however, exclude the possibility that an infinite-horizon limit is positive semidefinite. 
Stationary covariances of the uncontrolled covariance dynamics satisfy the Lyapunov equation $A\Sigma_{\rm s}+\Sigma_{\rm s}A\t+D=0$, see \eqref{final_cov_dyn}. 
If $A$ is Hurwitz, the Lyapunov equation admits a unique solution, which is positive-definite if and only if $(A,\Gamma)$ is controllable, otherwise, it is singular~\cite[Theorem~12.4]{hespanha2018linear}.
Thus, positive-definiteness of the covariance is preserved along finite-horizon feasible trajectories, while positive-semidefinite (singular) covariances may appear as infinite-horizon stationary limits.

The limiting analysis in the remainder of this section relies on two regularity properties of the fixed-final time optimal covariance/co-state trajectories $\{(\Sigma_{\tf}(\cdot),\Pi_{\tf}(\cdot))\}_{\tf>0}$. 
First, the optimal pair, and consequently the Hamiltonian map $\tf\mapsto\hamil_{\tf}(\cdot)$, depend continuously on the final time (Lemma~\ref{lem:reg:init_cost}). 
Second, the family of optimal pairs is uniformly bounded on compact subintervals of $\tf\in(0,+\infty)$ and, under Assumption~\ref{assump:point_limit}, is locally bounded in time uniformly in $\tf$, so that bounded limiting arcs can be extracted as $\tf\to+\infty$ (Proposition~\ref{cor:cov_costate_compactness}). 
Since these results are technical, their precise statements and proofs are relegated to Appendix~\ref{sec:appen:b}.

Next, we characterize the set associated with the asymptotic Hamiltonian-zeros and the corresponding limiting trajectories, and present the supporting results needed for the proof of the main theorem of this section, Theorem~\ref{thm:cov-steering}, which is stated in Section~\ref{subsect_V_D}.

\subsection{Asymptotic Gateway and Asymptotic Connections} \label{sect_5_subs_a}

In this section, we introduce the notion of Asymptotic Gateways (AG) and Asymptotic Connections (AC), and present the main result establishing the existence of ACs together with convergence of the limiting trajectories to the AG as the final time tends to infinity. 

\medskip

\begin{definitionbox}[ (Asymptotic Gateway)]\label{asympt_gateway_2}
Consider \Cref{prob:Cov:Steer} and \Cref{prob:Det:Cov:Steer} with their data and notations. We refer to the set
\begin{align}
\asmgate  \Let \left\{ \Sigma \in \mathbb{S}_{++}^n  \;\middle\vert\;  
\begin{array}{@{}l@{}}
A\Sigma + \Sigma A\t + D = 0 \text{ and} \\ ~~~\trace(Q\Sigma) = 0
\end{array}
\right\},\nn
\end{align}
as the \textbf{Asymptotic Gateway}.  
\end{definitionbox}

\medskip

\begin{definitionbox}[ (Asymptotic Connections)] \label{asymp_connection_cov}
A limiting trajectory $\Sigma(\cdot)$ obtained from the family of finite-horizon optimal covariance trajectories $\{\Sigma_{\tf}(\cdot)\}_{\tf>0}$ solving the fixed final time version of \Cref{prob:Det:Cov:Steer}, as $\tf\to+\infty$, is an \textbf{Asymptotic Connection} if there exists $t_1\in(0,+\infty]$ such that $\lim_{t\to t_1}\Sigma(t)=\Sigma_{\rm s}\in\mathcal G_\infty$.
\end{definitionbox}

The large-horizon picture is as follows. 
When $\varpi=0$ and the Asymptotic Gateway set $\mathcal{G}_\infty$ is nonempty, sufficiently large-horizon optimal covariance trajectories may approach $\mathcal{G}_\infty$, remain in its neighborhood for a long time interval, and depart from it, only to satisfy the prescribed terminal boundary condition. 
Recall from Remark~\ref{rem:inf_notions} that an asymptotic Hamiltonian-zero can be understood as a limiting property of finite-horizon optimal solutions, and does not necessarily imply the existence of an infinite-horizon optimal control. 
Accordingly, ACs describe the large-horizon structure of finite-horizon optimal covariance trajectories rather than classical infinite-horizon trajectories satisfying exactly the terminal covariance constraint.

In what follows, in view of the compactness argument from Proposition~\ref{cor:cov_costate_compactness}, Theorem~\ref{prop_exist_AC} formalizes the preceding turnpike-type large-horizon observation
by showing that, under Assumption~\ref{assump:point_limit} and $\mathcal G_\infty\neq\varnothing$, the family of finite-horizon optimal covariance trajectories of the fixed-final time version of \Cref{prob:Det:Cov:Steer} admits an Asymptotic Connection, and that the corresponding Hamiltonian converges to zero as $\tf\to+\infty$.

\begin{myOCP}
\begin{theorem} \label{prop_exist_AC}
Consider the fixed final time version of \Cref{prob:Det:Cov:Steer} with final time $\tf$ and $\varpi=0$.
Let Assumptions~\ref{assump:standing} and ~\ref{assump:point_limit} hold, and assume that $\mathcal{G}_\infty\neq\varnothing$.
Then, as $\tf\to+\infty$, the family $\aset[]{\Sigma_{\tf}(\cdot)}_{\tf>0}$ admits an Asymptotic Connection, and the corresponding Hamiltonian converges asymptotically to zero.
In other words, $\tf=+\infty$ is an asymptotic Hamiltonian-zero for \Cref{prob:Det:Cov:Steer} according to Definition~\ref{Def:HamZeros}.
\end{theorem}
\end{myOCP}

The (long) proof of Theorem~\ref{prop_exist_AC} (see Appendix \ref{sec:appen:b} for a detailed proof) is based on a compactness argument together with an invariance analysis of the limiting covariance/co-state trajectories associated with \Cref{prob:Det:Cov:Steer}.
For the readers' convenience, we first briefly summarize the main steps. We split the proof into five parts.
\begin{enumerate}[label={\textup{(\ref{prop_exist_AC}-\alph*)}}, leftmargin=*, widest=b, align=left]
\item In \ref{step:one}, we analyze the forward and backward limiting arcs from the family of finite-horizon optimal covariance/co-state trajectories using the Arzel\`a--Ascoli theorem \cite[Theorem 7.25]{rudin1976}.

\item In \ref{step:two}, using the fact that $\mathcal G_\infty\neq\varnothing$, we construct a feasible comparison trajectory that passes through the gateway $\mathcal G_\infty$ and show that the fixed-final time value function remains bounded for all sufficiently large $\tf$, which yields finiteness of the corresponding infinite-horizon running cost along the limiting arcs.

\item In \ref{step:three}, we apply LaSalle's invariance principle~\cite[Theorem~4.4]{ref_khalil} to the forward and reversed-time limiting trajectories and show that both limiting arcs approach the asymptotic gateway set $\mathcal G_\infty$ defined in Definition \ref{asympt_gateway_2}.

\item In \ref{step:four}, using equiboundedness, equi-Lipschitz continuity of the family of finite-horizon optimal covariance trajectories, and continuous dependence of the covariance/co-state pair on the final time, we show that this limiting behavior can be observed by finite-horizon optimal trajectories for sufficiently large $\tf$.

\item Finally, in \ref{step:five}, continuity of the Hamiltonian with respect to the final time implies that the Hamiltonian converges to zero asymptotically.
\end{enumerate}
\smallskip
Theorem~\ref{prop_exist_AC} addresses the case where the Asymptotic Gateway set $\asmgate$ is nonempty and shows that the limiting optimal trajectories admit an Asymptotic Connection.
The next remark clarifies why the existence of an Asymptotic Connection should not be interpreted as the attainment of an infinite-horizon covariance trajectory.

\begin{remark}[Attainment of an Asymptotic Connection] \label{rem:AC_attainment}
If $\Sigma_{\rm s}\in\mathcal{G}_\infty$ coincides with the prescribed terminal covariance, i.e., $\Sigma_{\rm s}=\fincov $, then the limiting trajectory is consistent with $\lim_{t\to+\infty}\Sigma(t)=\fincov $ and can be viewed as an attained infinite-horizon limit.
If, on the other hand, $\Sigma_{\rm s}\neq\fincov $, then the infinite-horizon solution cannot satisfy the prescribed terminal covariance condition. 
Indeed, at  $\asmgate$, the uncontrolled covariance dynamics satisfy $\dot\Sigma(\cdot)=0$, and hence the infinite-horizon covariance trajectory remains at $\Sigma_{\rm s}$ and cannot converge to a different $\fincov$. 
Thus, an Asymptotic Connection should be viewed only as a limiting covariance trajectory, but not as an attained one, consistent with Remark~\ref{rem:inf_notions}.
\end{remark}

We now turn to the case \(\asmgate=\varnothing\). 
While Theorem~\ref{prop_exist_AC} shows that $\asmgate \not= \varnothing$ implies an asymptotic Hamiltonian-zero, an empty gateway rules out this possibility. Specifically, if $\varpi=0$ and $\asmgate=\varnothing$, then, under Assumptions~\ref{assump:standing} and \ref{assump:point_limit}, one has $\lim_{\tf\to+\infty}\hamil_{\tf}(\cdot)>0$, and hence $\tf=+\infty$ cannot be an asymptotic Hamiltonian-zero (see Lemma~\ref{lem_empty_gateway} in Appendix~\ref{sec:appen:b}).

In the next subsection, we discuss the limiting behavior of Asymptotic Connections and conclude with a simple one-dimensional example illustrating the established concepts.

\subsection{Discussion on Asymptotic Connections} \label{subsect_V_D}

In this section, we interpret the results of Section~\ref{sect_5_subs_a} and discuss the limiting behavior induced by ACs.
We also explain why an asymptotic Hamiltonian-zero does not, in general, imply the existence of an admissible infinite-horizon optimal controller, and conclude with a simple one-dimensional example.

Whenever $\mathcal{G}_\infty\neq\varnothing$, under  Assumption \ref{assump:point_limit}, and from the compactness result of Proposition~\ref{cor:cov_costate_compactness}, a finite-horizon optimal solution of the fixed final time version of \Cref{prob:Det:Cov:Steer} can be \textit{approximated} by a trajectory that approaches an AG, remains in its neighborhood for some time,  and then departs to satisfy the prescribed terminal boundary condition. 
In this sense,  ACs provide useful insights into the behavior of sufficiently \emph{long-horizon optimal solutions}, similar to the classical \emph{turnpike phenomenon} in optimal control; see, e.g.,~\cite{anderson_turnpike}. 
Figure~\ref{fig:AC_turnpike_stitching} illustrates this turnpike interpretation of ACs.
\begin{figure*}[!t]
    \centering
    \includegraphics[scale=0.24]{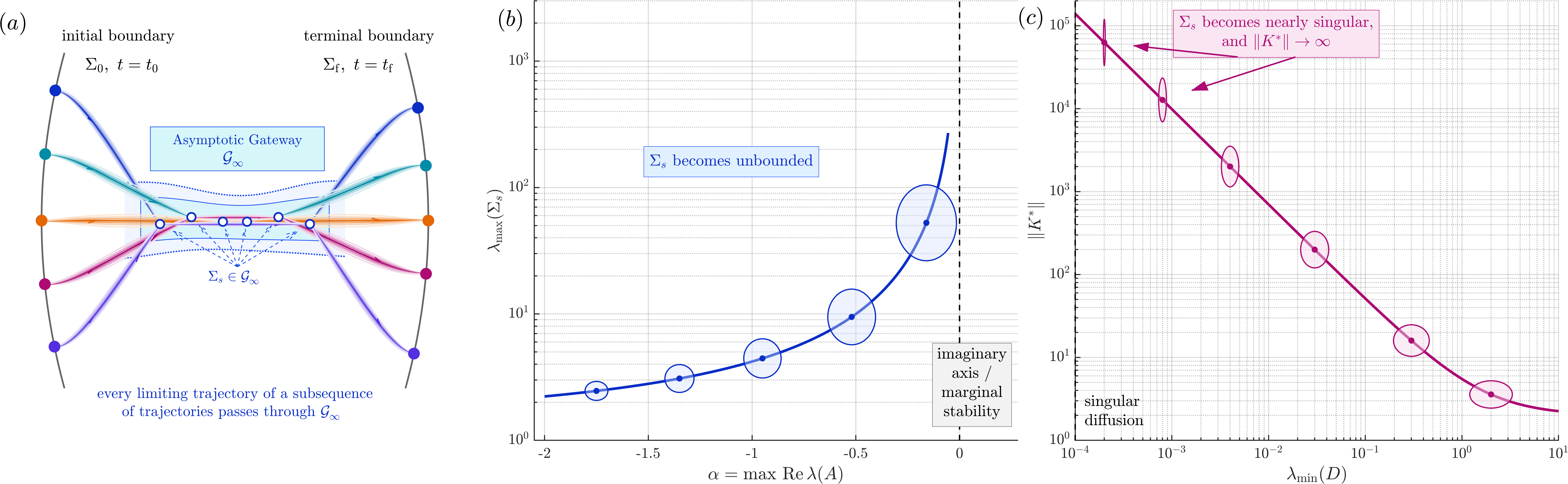}
    \caption{
    Asymptotic connections and limiting regimes of the Asymptotic Gateway $\mathcal{G}_{\infty}$.
    (a) illustrates finite-horizon covariance trajectories whose subsequential limiting trajectories pass through the Asymptotic Gateway $\mathcal{G}_{\infty}$ at stationary covariances $\Sigma_{\rm s}\in\mathcal{G}_{\infty}$.
    (b) illustrates the loss of boundedness of $\Sigma_{\rm s}$ as $A$ approaches marginal stability.
    (c) illustrates the approach of $\Sigma_{\rm s}$ to the boundary of $\mathbb{S}_{+}^{n}$, together with unbounded feedback control, as $D$ becomes nearly singular while $(A,\Gamma)$ remains controllable.
    } 
    \label{fig:asymp_connections}
\end{figure*}
Specifically, depending on the controllability and the spectral properties of the matrix pair \((A,\Gamma)\), the positive-definite stationary solution \(\Sigma_{\rm s}\) of the Lyapunov equation
\begin{align} 
\label{lyap_eqn}
A\Sigma_{\rm s}+\Sigma_{\rm s}A\t+D=0,
\end{align}
may approach the boundary of \(\mathbb S^n_+\).  
For example, when \(D\succcurlyeq0\), \(A\) is Hurwitz, and \(\trace(Q\Sigma_{\rm s})=0\), the matrix \(\Sigma_{\rm s}\) satisfies
$
    \lim_{t\to+\infty}\Sigma(t)=\Sigma_{\rm s}\in\mathcal G_\infty,
$
and $\Sigma_{\rm s}$ may even become singular at the limit. 
When the admissible trajectories of \eqref{final_ricc_dyn}--\eqref{final_cov_dyn} pass arbitrarily close to such singular covariances, the associated control effort may become unbounded near the singularity; see, e.g., \cite{sing_cov}. 

Another possibility is the occurrence of arbitrarily large covariances when some eigenvalues of $\Sigma_{\rm s}$ diverge, as some eigenvalues of $A$ approach the imaginary axis. %
Figure~\ref{fig:asymp_connections} illustrates three related limiting phenomena: the behavior associated with ACs, the loss of boundedness of $\Sigma_{\rm s}$ as $A$ approaches marginal stability, and the approach of $\Sigma_{\rm s}$ to the boundary of $\mathbb{S}_+^n$, together with unbounded feedback control as $D$ becomes nearly singular, while $(A,\Gamma)$ remains controllable.

\begin{example}
Consider the covariance steering \Cref{prob:Det:Cov:Steer} in a one-dimensional setting
with $Q \Let 0$, $\varpi \Let 0$, $D>0$, and $R>0$. 
The stationary covariance $\Sigma_{\rm s}=-{D}/{2A}$ satisfies the Lyapunov equation \eqref{lyap_eqn} 
and belongs to $\mathcal{G}_\infty$ when $A<0$. In this case, the covariance and co-state dynamics admit two candidate solution branches.
\begin{enumerate}[label={\textup{(\(\textsf{B}\)-\alph*)}}, leftmargin=*, widest=b, align=left]

\item \textbf{\textup{Zero co-state branch:}} When $\Pi(\cdot)\equiv 0$, the optimal feedback is $K^*(\cdot)\equiv 0$, and the covariance evolves according to
\begin{align*}
    \Sigma(t)=\Sigma_{\rm s} + (\initcov -\Sigma_{\rm s})e^{2At}\quad \text{for all }t \ge 0.
\end{align*}
For $A<0$, the trajectory converges asymptotically to $\Sigma_{\rm s}$ as $t\to+\infty$, and \eqref{total_hamiltonian} is satisfied.

\item \textbf{\textup{Nonzero co-state branch:}}
When $\Pi(t)\neq 0$ for some $t\ge 0$, the transversality condition \eqref{total_hamiltonian} yields
\begin{align*}
    \Pi(t)=\frac{R}{B^2}\left(2A+\frac{D}{\Sigma(t)}\right)\quad \text{for all }t \ge 0,
\end{align*}
while the covariance evolves as
\begin{align*}
    \Sigma(t)=\Sigma_{\rm s} + (\initcov -\Sigma_{\rm s})e^{-2At}\quad \text{for all }t \ge 0.
\end{align*}
The corresponding optimal feedback is 
\begin{align*}
    K^*(t)=-\frac{1}{B}\left(2A+\frac{D}{\Sigma(t)}\right)\quad \text{for all }t \ge 0.
\end{align*}
For $A<0$, this trajectory moves away from $\Sigma_{\rm s}$ in forward time, while converging asymptotically to $\Sigma_{\rm s}$ in backward time.

\end{enumerate}
\medskip
These two branches correspond to the stable and unstable manifolds of the Hamiltonian system about the equilibrium $(\Sigma_{\rm s},0)$ associated with the covariance and co-state dynamics while satisfying the transversality condition \eqref{total_hamiltonian}.

Assume that $\initcov $ and $\fincov $ lie on the opposite sides of $\Sigma_{\rm s}$. In this case, neither branch yields a finite-time trajectory connecting $\initcov $ and $\fincov $. Indeed, the zero co-state branch approaches $\Sigma_{\rm s}$ asymptotically and cannot cross it in finite time, whereas the nonzero co-state branch moves away from $\Sigma_{\rm s}$ in forward time.

Consequently, the only feasible Hamiltonian-zero solution is asymptotic, i.e., $\tfc=+\infty$. The resulting solution consists of two asymptotic segments: a forward uncontrolled trajectory from $\initcov $ to $\Sigma_{\rm s}$, and a backward controlled trajectory from $\fincov $ to $\Sigma_{\rm s}$, which meet only in the limit.
Moreover, once both trajectories reach the gateway $\Sigma_{\rm s}$, one has $\Pi=K^\ast=\dot{\Sigma}=0$,
and therefore the system cannot depart from $\Sigma_{\rm s}$ to satisfy a terminal condition on the opposite side. Hence, no admissible infinite-horizon optimal controller exists in this case.
On the other hand, if $\fincov =\Sigma_{\rm s}$, then the uncontrolled trajectory with $K^*(\cdot)\equiv0$ provides an admissible infinite-horizon optimal solution.

\end{example}

Having established existence conditions for both finite and asymptotic Hamiltonian-zeros, it remains to record the asymptotic limits of the Hamiltonian $\hamil_{\tf}(\cdot)$ as the final time approaches zero and infinity. 
These limits are established in Lemma~\ref{inf_hor_cs} in Appendix~\ref{sec:appen:b}: as $\tf\to 0^+$, one has $\hamil_{\tf}(\cdot)\to-\infty$, whereas the large-horizon limit $\hamil^\infty(\cdot) \Let \lim_{\tf \to +\infty} \hamil_{\tf}(\cdot)$ exists and satisfies $\hamil^\infty(\cdot)\ge \varpi/2$, with strict inequality whenever $\mathcal{G}_\infty=\varnothing$. 
With these technical ingredients in place, we are now ready to state the main theorem of this section. Its proof relies on Theorem~\ref{prop_exist_AC} with the supporting results of Appendix~\ref{sec:appen:b} (Proposition~\ref{cor:cov_costate_compactness}, Lemma~\ref{lem_empty_gateway}, and Lemma~\ref{inf_hor_cs}), and is presented in Appendix~\ref{sec:appen:b}.

\begin{myOCP}
\begin{theorem}
\label{thm:cov-steering}
Consider the covariance steering \Cref{prob:Det:Cov:Steer} and let Assumptions~\ref{assump:standing} and \ref{assump:point_limit} hold. 
Then, an asymptotic Hamiltonian-zero exists if and only if $\varpi=0$ and, furthermore, there exists $\Sigma_{\rm s} \succ 0$, such that $A\Sigma_{\rm s} + \Sigma_{\rm s} A\t + D = 0$ and $\trace(Q\Sigma_{\rm s}) = 0$.
\end{theorem}
\end{myOCP}

The next result is an immediate consequence of Theorem~\ref{thm:cov-steering}.

\begin{myOCP}
\begin{corollary}
Consider the covariance steering \Cref{prob:Det:Cov:Steer} and let Assumption \ref{assump:point_limit} hold. If $Q \succcurlyeq 0$ with $Q\neq0$ or $\varpi >0$, then there exists at least one finite Hamiltonian-zero.
\end{corollary}
\end{myOCP}

We note that the strict positivity of $\hamil^\infty(\cdot)$, together with  $\lim_{\tf \to 0^+} \hamil_{\tf}(\cdot)=-\infty$ and the continuity of the Hamiltonian with respect to the final time
from Lemma~\ref{lem:reg:init_cost},  yields a finite Hamiltonian-zero via the intermediate value theorem.

%% file: sections/4.1-Numerics_snsvty.tex
\section{Final time Sensitivities and Numerical Algorithm} \label{Section_7}

This section provides a computational bridge from the theory to an implementable numerical solver. 
Thus far, we have characterized optimal final-time candidates as zeros of the fixed-final-time Hamiltonian and identified the asymptotic mechanism by which such zeros may escape to infinity. 
We now turn this characterization into a {numerical procedure}. 
To this end, we first derive the sensitivities of the Hamiltonian map $\tf \mapsto \hamil_{\tf} (\cdot)$ with respect to the final time. We then employ these sensitivities in a \emph{trust-region line-search} method that searches for finite Hamiltonian-zeros while monitoring proximity to the Asymptotic Gateway in order to detect asymptotic Hamiltonian-zero behavior.

\subsection{Final time Sensitivity of the Hamiltonian} \label{subsect7_sens}

In general, an analytical expression for the co-state sensitivity is not available.  However, for a special class of diffusion matrices, both the co-state sensitivity and the Hamiltonian sensitivity {can be characterized analytically}.

To this end, fix a final time $\tf$ and let \(\eps>0\). 
Consider the perturbation ${\tfe}\Let \tf+\eps \delta \tf$. Let $\chi_{\tf}(\cdot)\Let \bigl(\Sigma_{\tf}(\cdot),\Pi_{\tf}(\cdot)\bigr)$ and $\chi_{\tfe}(\cdot)\Let\bigl(\Sigma_{\tfe}(\cdot),\Pi_{\tfe}(\cdot)\bigr)$ denote the corresponding optimal covariance/co-state pairs for unperturbed and perturbed trajectories, respectively. 
By construction, for every $\eps>0$, we have $\Sigma_{\tfe}(0)=\Sigma_{\tf}(0)=\initcov $ and $\Sigma_{\tf^\eps}(\tf^\eps)=\Sigma_{\tf}(\tf)=\fincov $.

Define the endpoint variations induced by the perturbation $\delta \tf$ as
\begin{align}
\delta  \hamil_{\tf} &\Let \lim_{\eps\to 0} \frac{\hamil_{\tfe}\bigl(\chi_{\tfe}(\tfe)\bigr) - \hamil_{\tf}\bigl(\chi_{\tf}(\tf)\bigr)}{\eps}, \label{cov_sens_pert_1} \\
\delta \Pi_{\tf} &\Let \lim_{\eps\to 0} \frac{ \Pi_{\tfe}(\tfe)-\Pi_{\tf}(\tf)
}{\eps}, \label{cov_sens_pert_2} \\
\delta \Pi_{\tf;t} &\Let \lim_{\eps \to 0} \frac{\Pi_{\tfe}(t)-\Pi_{\tf}(t)}{\eps} \label{cov_sens_pert_3} \quad \text{for all } t\in[0,\tf].
\end{align}
Then, following Equations \eqref{final_ham_dyn}, \eqref{final_ricc_dyn} and \eqref{final_cov_dyn}, the endpoint sensitivity of $\hamil_{\tf}(\cdot)$ with respect to $\tf$, denoted by $\tfrac{\delta \hamil_{\tf}}{\delta \tf}(\cdot)$, is related to the covariance dynamics and the initial value (respectively, endpoint) sensitivity of the co-state $\tfrac{\delta {\Pi}_{\tf;0}}{\delta \tf}$ (respectively, $\frac{\delta {\Pi}_{\tf}}{\delta \tf}$) through 
\begin{align}\label{tilde_H_sens}
    \hspace{-0.5 em}\frac{\delta {\hamil}_{\tf}}{\delta \tf}(\cdot) = \half \trace \left(\frac{\delta {\Pi}_{\tf;0}}{\delta \tf} \dot{\Sigma}(0) \right)=\half \trace \left(\frac{\delta {\Pi}_{\tf}}{\delta \tf} \dot{\Sigma}(\tf) \right)\!.
\end{align}
In general, there is no analytical expression to compute the co-state sensitivities. 
However,
for the special case $D=\kappa B B\t$ for some scalar $\kappa \geq 0$, the initial co-state $\Pi_0$ can be computed analytically from \cite{cite3_3} for prescribed boundary conditions $\initcov ,\fincov \in \mathbb{S}_{++}^n$. Namely, from \eqref{ham_matr} and \eqref{STM_matr}, the analytical expression for $\Pi_0$ is
\begin{align}
    &\hspace{-.7em}\Pi_0 = -\Phi_{12}(\tf,0)^{-1} \Phi_{11}(\tf,0) +\initcov ^{-\half} \bigl(\tfrac{\kappa}{2} I_n-\Xi\bigr)\initcov ^{-\half}, \label{p3_cont_analyt} \\
    &\hspace{-.7em}\Xi = \hspace{-1mm} \left[\frac{\kappa^2}{4} I_n +\initcov ^{1/2}\Phi_{12}(\tf,0)^{-1}\fincov \Phi_{12}(\tf,0)\t\initcov ^{1/2} \right]^{1/2}.
\end{align}
Then,  for $t\in[0,\tf]$, $\Pi(t), \Sigma(t)$ can also be computed analytically; see \cite{cite3_3}.
In this case, the co-state sensitivity $\frac{\delta {\Pi}_{\tf;0}}{\delta \tf}$ can be obtained by differentiating \eqref{p3_cont_analyt} with respect to $\tf$, yielding
\begin{align} \label{tilde_H_sens_0}
    &\frac{\d {\Pi}_0}{\d \tf} 
    = \Phi_{12}(\tf,0)^{-1}\frac{\d \Phi_{12}}{\d \tf}(\tf,0)  \Phi_{12}(\tf,0)^{-1} \Phi_{11}(\tf,0) \nonumber \\
    & - \Phi_{12}(\tf,0)^{-1}\frac{\d \Phi_{11}}{\d \tf}(\tf,0) - \initcov ^{-\half} \Y \initcov ^{-\half} , \\
    &= -\Phi_{12}(\tf,0)^{-1} B R^{-1} B\t \bigl(\Phi_{12}{(\tf,0)^{-1}}\bigr)\t-\initcov ^{-\half} \Y \initcov ^{-\half},\nn
\end{align}
where $\Y$ is a solution to the Lyapunov equation
\begin{align} \label{sylv}
    \Xi \Y \hspace{-0.3mm} + \hspace{-0.3mm} \Y \Xi = -\mathcal{S} \hspace{-0.3mm} \left(\initcov ^{1/2}\Phi_{12}(\tf,0)^{-1}\frac{\d \Phi_{12}}{\d \tf}(\tf,0)  \Phi_{12}(\tf,0)^{-1} 
    \fincov  \bigl({\Phi_{12}(\tf,0)^{-1}}\bigr)\t\initcov ^{1/2} \right).
\end{align}
Equation \eqref{sylv} can be solved using the well-known Kronecker formula; see \cite[Appendix A]{anderson}. 
Equation~\eqref{tilde_H_sens} suggests an important property of Asymptotic Connections, i.e., their relation to critical points of the value function with respect to the final time, at which both the first and second derivatives vanish, as stated next.

\begin{remark}\label{prop_ac_infty}
    Under Assumption \ref{assump:point_limit}, it follows from Theorem \ref{prop_exist_AC} that, along an Asymptotic Connection with $\tf \to +\infty$, the associated covariance trajectories approach $\mathcal{G}_\infty$ with $\lim_{t \to +\infty} B\t \Pi (t)\to 0$, $\lim_{t \to \infty}\dot{\Sigma}(t)=0$ and $\lim_{\tf \to +\infty} \rm H_{\tf}(\cdot)=0$. 
    In that case,~\eqref{tilde_H_sens} implies that $\lim_{\tf\to+\infty}\frac{\delta \rm H_{\tf}}{\delta \tf}(\cdot)=0$.
    Therefore, Asymptotic Connections are associated with an asymptotic Hamiltonian-zero regime in which, together with Corollary~\ref{cor_dv_H_0}, both the first and second derivatives of the value function with respect to the final time vanish as $\tf\to+\infty$.
\end{remark}

\subsection{Trust-region Based Line Search Algorithm} \label{subsect7:alg}

Building on our analysis in Section \ref{Section_5}, we compute the candidate optimal final times by minimizing the value function associated with the cost functional \eqref{eq:cost} with respect to the final time $\tf$. The resulting optimization problem is addressed using a trust-region line-search algorithm, following the framework in~\cite{nocedal2006}.

\begin{algorithm}
\caption{\texttt{Trust-Region-Based Final Time Optimization}}
\label{alg_1}
\begin{algorithmic}[1]
\Statex \hspace{-\algorithmicindent} \textbf{Input:} Initial final time $\tf^{(0)}$, bounds $0<t_{\rm f,min}<t_{\rm f,max}$, acceptance threshold $\eta\in[0,0.25)$, initial trust-region radius $\Delta^{(0)} > 0$, trust-region bounds $\Delta_{\rm min},\Delta_{\rm max} > 0$, step-size control constants $\gamma_{\rm inc} > 1$, $\gamma_{\rm dec} \in (0,1)$, and convergence tolerance $\eps > 0$
\Statex \hspace{-\algorithmicindent} \textbf{Output:} Candidate optimal final time $\tfc$ and output flag $\dagger$
\Statex \hspace{-\algorithmicindent} \textbf{Initialize:}
\State $k\gets 0$, set an initial bracket $[t_{\rm low},t_{\rm high}]\subseteq[t_{\rm f,min},t_{\rm f,max}]$ if available
\State $\big(\hamil^{(0)}, \delta \hamil^{(0)}_{\tf^{(0)}}/\delta \tf, \cdot \big) \gets \texttt{CS}(\tf^{(0)})$
\While{1}
    \If{\texttt{isempty}$\big(\delta \hamil^{(k)}_{\tf^{(k)}}/\delta \tf\big)$}
         \If{$k = 0$}
            \State $\delta \hamil^{(k)}_{\tf^{(k)}}/\delta \tf \gets 1$
         \Else
             \State $\delta \hamil^{(k)}_{\tf^{(k)}}/\delta \tf \gets \frac{\big(\hamil^{(k)} - \hamil^{(k-1)}\big)}{\big(\tf^{(k)} - \tf^{(k-1)}\big)}$
         \EndIf
    \EndIf
    \State $\mathcal{B}^{(k)} \gets \delta \hamil^{(k)}_{\tf^{(k)}}/\delta \tf$

    \If{$|\hamil^{(k)}| < \eps \ \land\ \mathcal{B}^{(k)} > 0$}
        \State \Return $\tf^{(k)}, \ \dagger \gets 1$
    \EndIf

    \State Update $[t_{\rm low},t_{\rm high}]$ using the sign of $\hamil^{(k)}$

    \If{$\mathcal B^{(k)}>0$ and $|\mathcal B^{(k)}|$ is sufficiently large}
        \State $\delta^{(k)}\gets
        \mathrm{clip}\big(-\hamil^{(k)}/\mathcal B^{(k)},-\Delta^{(k)},\Delta^{(k)}\big)$
    \Else
        \State $\delta^{(k)}\gets -\Delta^{(k)}\operatorname{sgn}\big(\hamil^{(k)}\big)$
    \EndIf
    
        \State $\widehat{\tf}\gets
    \mathrm{clip}\big(\tf^{(k)}+\delta^{(k)},t_{\rm f,min},t_{\rm f,max}\big)$
    \State If a bracket is available, clip $\widehat{\tf}$ to $[t_{\rm low},t_{\rm high}]$
    \If{$|\widehat{\tf}-\tf^{(k)}|<\epsilon$}
        \State $\widehat{\tf}\gets (t_{\rm low}+t_{\rm high})/2$
    \EndIf

    \State $\big(\widehat{\hamil},\delta\widehat{\hamil}_{\widehat{\tf}}/\delta\tf,\widehat\Sigma\big)
    \gets \texttt{CS}(\widehat{\tf})$
    \State $\Upsilon\gets \texttt{Detector}(\widehat\Sigma,\widehat{\hamil})$

    \If{$\Upsilon=1$}
        \State Mark $\widehat{\tf}$ as an asymptotic side of the search interval
        \State Set $\tf^{(k+1)}\gets (t_{\rm low}+t_{\rm high})/2$
        \State $\big(\hamil^{(k+1)},\delta\hamil^{(k+1)}_{\tf^{(k+1)}}/\delta\tf,\Sigma^{(k+1)}\big)
        \gets \texttt{CS}(\tf^{(k+1)})$
    \Else
        \State $s^{(k)}\gets \widehat{\tf}-\tf^{(k)}$
        \State $\Delta m^{(k)}\gets
        -\hamil^{(k)}s^{(k)}
        -\frac{1}{2}\mathcal B^{(k)}\big(s^{(k)}\big)^2$
        \State $\rho^{(k)}\gets
        |\widehat{\hamil}-\hamil^{(k)}|/|\Delta m^{(k)}|$
        \State Update $\Delta^{(k+1)}$ from $\rho^{(k)}$ via \cite{nocedal2006}
        \If{$\rho^{(k)}<\eta$}
            \State $\tf^{(k+1)}\gets \tf^{(k)}$
        \Else
            \State $\tf^{(k+1)}\gets \widehat{\tf}$
        \EndIf
        \State $\big(\hamil^{(k+1)},\delta\hamil^{(k+1)}_{\tf^{(k+1)}}/\delta\tf,\Sigma^{(k+1)}\big)
        \gets \texttt{CS}(\tf^{(k+1)})$
    \EndIf
    \State $k\gets k+1$
\EndWhile
\end{algorithmic}
\end{algorithm}

\begin{algorithm} 
\caption{\texttt{Detector}}
\label{alg_2}
\begin{algorithmic}[1]
\Statex \hspace{-\algorithmicindent} \textbf{Input:} Covariance trajectory  $ {\Sigma}_{\rm new}$, Hamiltonian $\hamil$, Threshold $\eta$
\Statex \hspace{-\algorithmicindent} \textbf{Output:} Flag parameter $\Upsilon$
\Statex \hspace{-\algorithmicindent} \textbf{Initialize:} $\Sigma(\cdot) \gets  {\Sigma}_{\rm new}$
\If{$\texttt{any}\big(\Sigma(\cdot) \in \mathcal{G}_\infty\big) \ \land \ (|\hamil|<\eta)$}
    \State$\Upsilon \gets 1$
\Else
    \State$\Upsilon \gets 0$
\EndIf

\end{algorithmic}
\end{algorithm}

Algorithm~\ref{alg_1} locates the finite zeros of the mapping $\tf \mapsto \hamil_{\tf}(\cdot)$ by iteratively updating the final time. Broadly, it consists of the following steps: 
\begin{enumerate}[label={\textup{\alph*)}}, leftmargin=*, widest=b, align=left]
\item At each iteration, the covariance steering problem is solved via the module \texttt{CS} in order to evaluate the Hamiltonian $\hamil^{(k)}$ and its final time sensitivity, given by \eqref{tilde_H_sens}, analytically when $D=\kappa BB\t$ for some $\kappa\ge 0$, and numerically otherwise. 
\item The Hessian at iteration $k$, denoted as $\mathcal{B}^{(k)}$, approximates the final time sensitivity of the Hamiltonian, $\frac{\d \hamil_{\tf}}{\d \tf}$, and therefore has the role of a Hessian approximation for the associated value function. 
It is computed either by using secant-based Hessian estimation or, when available, by analytically computing the Hamiltonian sensitivities (see Section~\ref{subsect7_sens}).
\item At each update, the \texttt{Detector} module in Algorithm~\ref{alg_2} monitors whether the computed covariance trajectory approaches the asymptotic gateway $\mathcal{G}_\infty$ and whether the Hamiltonian satisfies $|\hamil^{(k+1)}|<\eta$ for some user-defined threshold $\eta>0$. Its role is to prevent the solver from slowing down or getting stuck at large horizons while attempting to converge to an asymptotic Hamiltonian-zero identified in Section~\ref{Section_5}.
\item Detection of an asymptotic Hamiltonian-zero
is performed using the \texttt{any} function in Line~1 of the Algorithm~\ref{alg_2},
which returns $1$ whenever at least one sampled point along the trajectory is sufficiently close to the gateway set $\mathcal{G}_\infty$.
\end{enumerate}
In practice, proximity to $\mathcal{G}_\infty$ can be assessed at iteration $(k+1)$ through a tolerance-based test using the conditions
\begin{align*}
\|A\Sigma^{(k+1)}(s)+\Sigma^{(k+1)}(s) A\t +D\| \le \eta_{\rm G} \,\text{ and }\,
|\trace(Q\Sigma^{(k+1)}(s))|\le \eta_{\rm G},
\end{align*}
for some user-defined tolerance $\eta_{\rm G}>0$, evaluated at sampled points $s\in[0,\tf^{(k+1)}]$ along the computed trajectory $\Sigma^{(k+1)}(\cdot)$. 
If a detection occurs, $\tf$ is reduced within the trust region until either convergence to a finite final time is achieved or the lower bound $t_{\rm f,min}$ is reached, in which case an asymptotic Hamiltonian-zero is declared. 

The algorithm terminates (Line~13)  when the Hamiltonian at iteration $k$ satisfies $|\hamil^{(k)}|<\eps$ for some threshold $\eps>0$, and $\mathcal{B}^{(k)}$ is positive. 

%% file: sections/4.2-Numerics_algo.tex
\section{Numerical Examples} \label{Section_8}

In this section, we present three different numerical examples to illustrate our findings and demonstrate Algorithm~\ref{alg_1}. 
First, we illustrate the effect of the noise matrix $\Gamma$ in \eqref{StochasticSystem} on the optimal final time $\tf^*$ of \Cref{prob:Cov:Steer}. 
Then, Algorithm~\ref{alg_1} is used to solve \Cref{prob:Cov:Steer} to design a spacecraft maneuver under the linear Clohessy-Wiltshire-Hill (CWH) dynamics. Lastly, we introduce the notion of the \textit{price of synchronization} for a simple Gaussian mixture model (GMM) steering problem.

\begin{figure*}[t]
    \centering
    \includegraphics[scale=0.45]{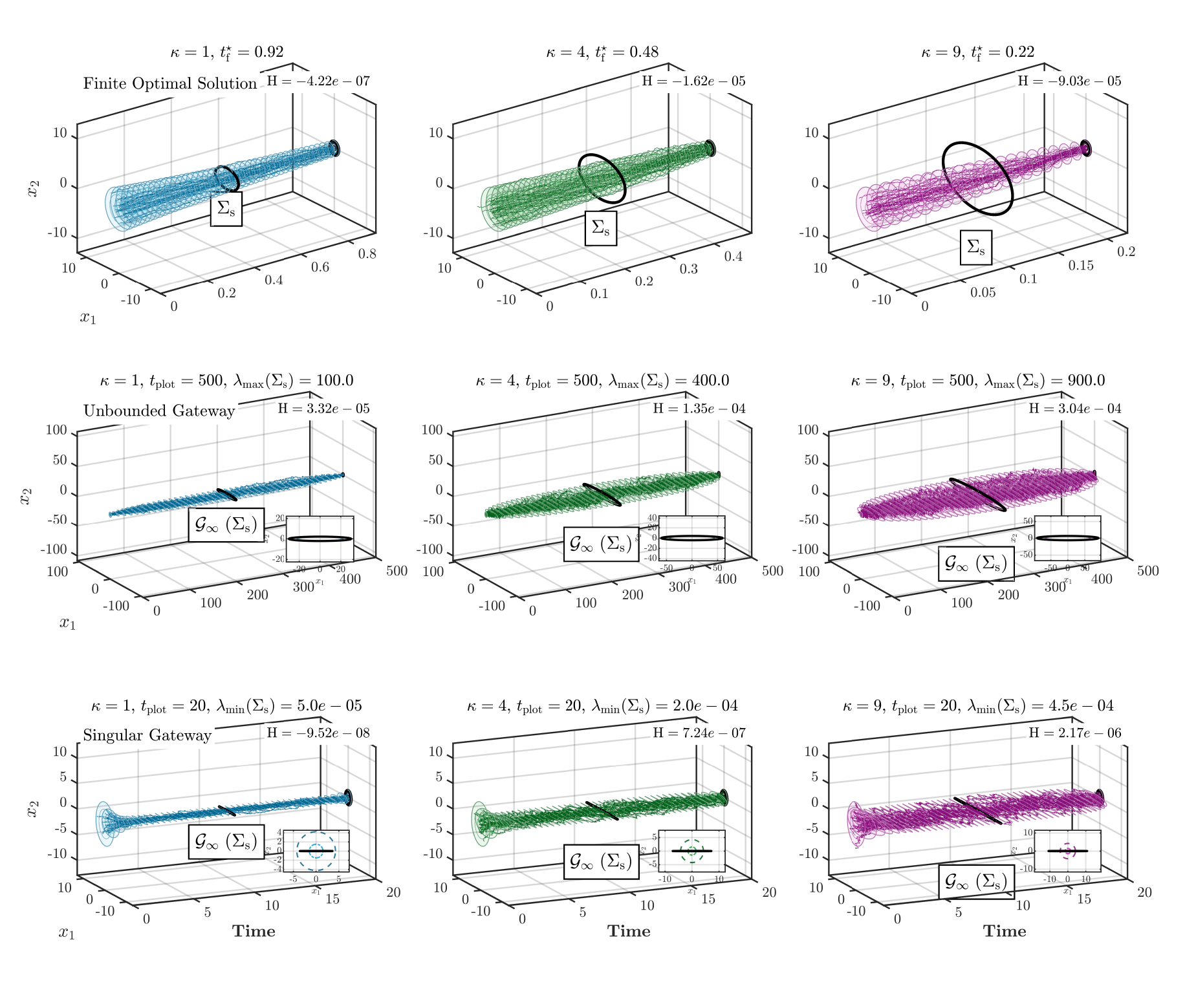}
    \caption{Numerical example illustrating the cases of finite optimal final time, unbounded gateways, and singular gateway in free-final time covariance steering problems.   
    Each panel shows the covariance trajectory over time together with sample paths of the associated closed-loop diffusion. 
    The first row ($Q=I$) exhibits finite optimal final time satisfying the transversality condition $\hamil_{\tf^*}(\cdot) = 0$. 
    The second row ($Q=0$ and $\alpha=0.005$) results in a large Lyapunov covariance $\Sigma_{\rm s}$, with $\lambda_{\max}(\Sigma_{\rm s})=100$, $400$, and $900$ as $\kappa$ increases. 
    The third row ($Q=0$ and $\delta_B=10^{-4}$) results in a nearly singular Lyapunov covariance with small $\lambda_{\min}(\Sigma_{\rm s})$.}
    \label{fig1}
\end{figure*}

\subsection{Effect of Noise Intensity on the Optimal Final Time and Asymptotic Gateways}

We start with a simple 2D example to illustrate the covariance steering problem with three distinct behaviors of the free-final time problem: convergence to a finite optimal final time, convergence toward an unbounded asymptotic gateway, and convergence toward a singular asymptotic gateway. We consider the SDE \eqref{StochasticSystem}  with the parameters
\begin{align*}
A \Let     \begin{bmatrix}
        -\alpha & 0\\
        0 & -1
    \end{bmatrix},\quad  
B \Let     \begin{bmatrix}
        1 & 0\\
        0 & \sqrt{\delta_B}
    \end{bmatrix}, \quad R=I,
\end{align*}
and boundary covariances $\initcov \Let 2I$ and $\fincov \Let I/4$. 
We set the diffusion matrix to $D \Let  \kappa B B\t$ and vary the noise intensity $\kappa \in \{1,4,9\}$ with $\varpi=0$ to examine its effect on the optimal common final time.
We consider both $Q=I$ and $Q=0$ to illustrate the asymptotic connection behavior. 
For each choice of $(\alpha,\delta_B,\kappa)$, we also compute the stationary covariance $\Sigma_{\rm s}$ associated with the uncontrolled diffusion, namely the solution of the Lyapunov equation
\begin{align} \label{eq:sec_num_lyap}
    A\Sigma_{\rm s}+\Sigma_{\rm s}A \t+\kappa BB\t=0.
\end{align}

Figure~\ref{fig1} illustrates the covariance evolution for different noise levels $\kappa$ along with several trajectory realizations. 
The colored ellipses show the $3\sigma$ covariance level sets along the trajectories, whereas the thick black covariance ellipse corresponds to the stationary Lyapunov covariance $\Sigma_{\rm s}$ \eqref{eq:sec_num_lyap}. 
In the second and third rows, this ellipse represents the asymptotic gateway $\Sigma_{\rm s} \in \mathcal G_\infty$ approached by the finite-horizon covariance trajectories.

Higher noise intensity leads to an earlier optimal final time in the finite optimal final time case, since stronger diffusion accelerates covariance growth and therefore requires the controller to act sooner to minimize the accumulated cost.
This behavior is shown in the first row, where $Q=I$, $\alpha=0.35$, and $\delta_B=1$. The computed optimal final times decrease from $t_{\rm f}^{\star}=0.92$ to $0.48$ and $0.22$ as $\kappa$ increases from $1$ to $4$ and $9$, respectively.
The displayed Hamiltonian values confirm that the transversality condition $\hamil_{\tf}(\cdot)=0$ is satisfied up to numerical tolerance.

The second row in Figure~\ref{fig1} shows an unbounded gateway regime. In this case, we consider $Q=0$, $\alpha=0.005$, and $\delta_B=1$.
Since the slow mode approaches marginal stability as $\alpha$ becomes small, the solution to the Lyapunov equation \eqref{eq:sec_num_lyap} becomes large.
Hence, the covariance trajectory approaches a limiting covariance whose size grows with the noise intensity and becomes unbounded as $\alpha\to 0^+$.

The third row shows a singular gateway regime. Here, we take $Q=0$, $\alpha=0.35$, and $\delta_B=10^{-4}$.
Although the pair $(A,B)$ remains controllable, the diffusion in the second coordinate is nearly singular. Consequently, the covariance solution for the Lyapunov equation \eqref{eq:sec_num_lyap} has a very small eigenvalue.
The resulting covariance tube approaches a thin limiting covariance, again marked by the thick black gateway ellipse.

\subsection{Covariance Steering of a Spacecraft}

In this example, we consider a spacecraft maneuver in which the covariance of a deputy spacecraft is steered toward a target covariance. 
The system dynamics are modeled using the CWH equations \cite[Section~7.4]{curtis2020orbital}, assuming a circular orbit at 800 km altitude, which corresponds to the LTI system \eqref{StochasticSystem} with parameters
\begin{align*}
    A= \hspace{-0.5mm} \begin{bmatrix}
0_{3} & I_{3}\\
\omega^{2}\,\begin{bmatrix}3&0&0\\[-2pt]0&0&0\\[-2pt]0&0&-1\end{bmatrix} & 2\omega\!\begin{bmatrix}0&1&0\\[-2pt]-1&0&0\\[-2pt]0&0&0\end{bmatrix}
\end{bmatrix},
\quad
B=m_{\rm d}^{-1}\begin{bmatrix}0_{3}\\ I_{3}\end{bmatrix} \hspace{-0.9mm},
\end{align*}
where $m_{\rm d} = 300$ kg is the mass of the spacecraft, $\omega = \sqrt{\mu_{\rm E}/r}$ is the orbital frequency with gravitational parameter of Earth $\mu_{\rm E}$ and the radius of the orbit of the spacecraft $r$.
We take $\kappa=225$ with $R=10^3 \ {\rm I}_3$, $Q=\blkdiag \{ 10 \, {\rm I}_3,{\rm I}_3 \}$, $\Sigma_0=\blkdiag \{ 10 \, {\rm I}_3,{\rm I}_3 \}$, $\Sigma_{\rm f}=\Sigma_0/10$.

The covariance evolution obtained from the optimal steering controller computed jointly with the optimal final time is shown in Figure~\ref{fig:spacecraft_state}. 
Both position and velocity covariance ellipsoids are visualized along an artificial time-flow direction to avoid overlap. 
It can be seen that the controller first reduces and reshapes the velocity uncertainty and then the position uncertainty to meet the prescribed terminal covariance constraints at the optimal final time. 

The convergence behavior of Algorithm~\ref{alg_1} is shown in Figure~\ref{fig:opt_logs} for two different initial guesses of the final time.
In the low-initialization case for the free-final time in
Figure~\ref{fig:opt_logs}(a), the final time increases monotonically and the Hamiltonian approaches zero from below.
The high-initialization case for the free-final time, in
Figure~\ref{fig:opt_logs}(b) is more challenging. 
The iterates initially remain near a flat region of the Hamiltonian, where the first-order sensitivity of the Hamiltonian with respect to the final time is nearly zero.
In this regime, a purely local update either stagnates or makes negligible progress. However, the proposed trust-region update Algorithm~\ref{alg_1} successfully drives the iterates away from this region.
After escaping the plateau, the Hamiltonian rapidly approaches zero and the algorithm converges to the same finite optimal final time obtained from the low-initialization case, illustrating the robustness of Algorithm~\ref{alg_1} with respect to the initial final time guess.

\begin{figure*}[t]
\centering
\begin{minipage}[t]{\linewidth}
    \centering
    \includegraphics[scale=0.25]{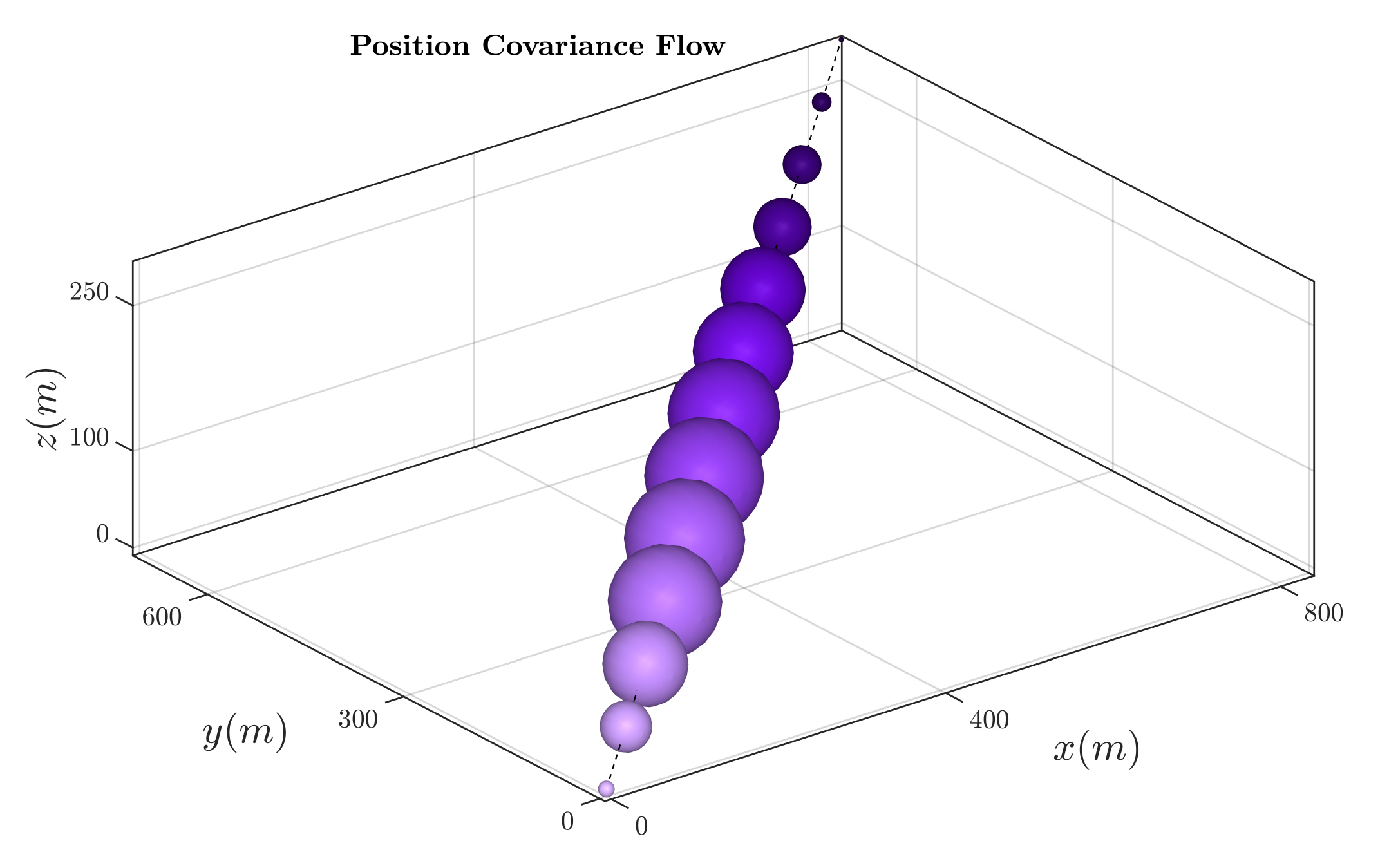}
\end{minipage}
\begin{minipage}[t]{\linewidth}
    \centering
    \includegraphics[scale=0.25]{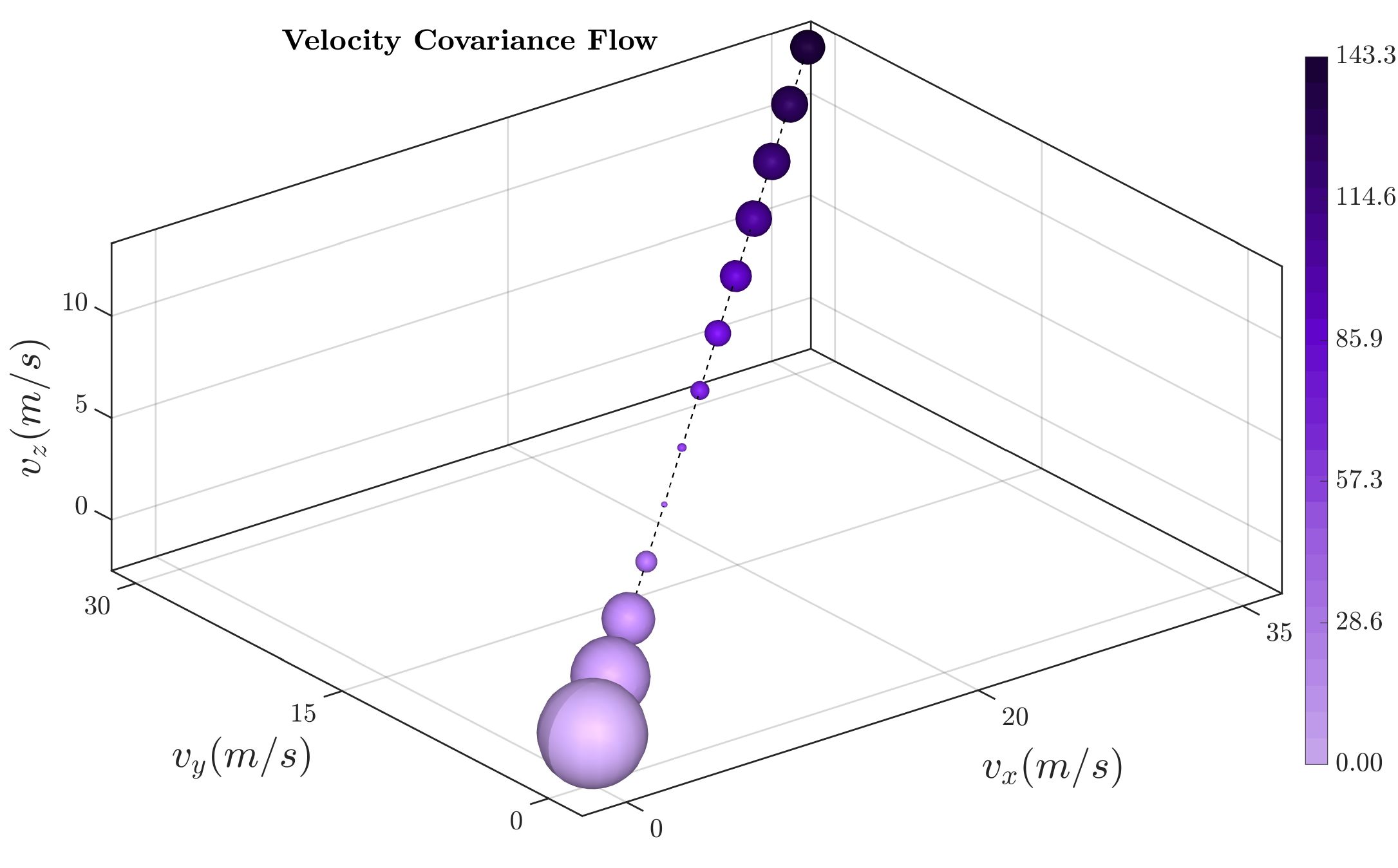}
\end{minipage}
\caption{The position and velocity distributions are shown in terms of their respective means and covariances within $3\sigma$ of the uncertainty.
The left panel shows the position covariance flow, while the right panel shows the velocity covariance flow, with trajectories evolving from lighter to darker shades over time.}
\label{fig:spacecraft_state}
\end{figure*}

\begin{figure}[t]
\begin{center}
        \centering
        \includegraphics[scale=0.6]{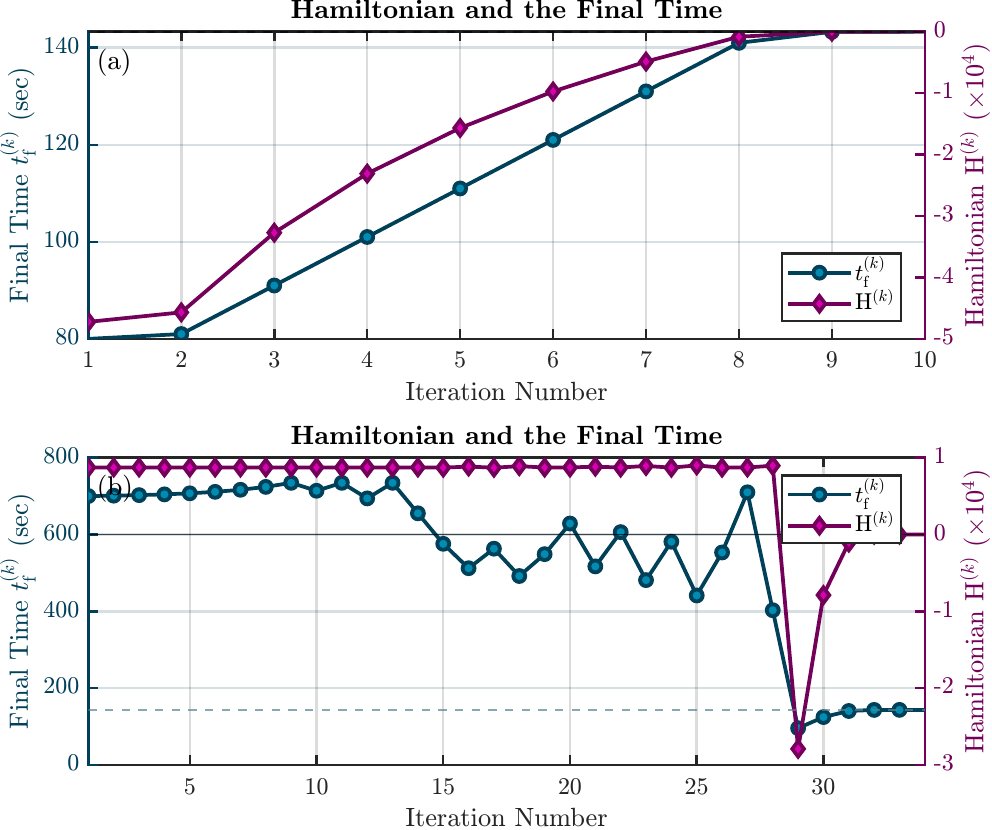}
        \caption{Convergence of the optimization algorithm for the covariance steering problem initialized with both a low and a high initial guess.}
        \label{fig:good_init}
    \label{fig:opt_logs}
\end{center}
\end{figure}

\subsection{The Price of Synchronization}

Lastly, we aim to steer a Gaussian mixture $\rho_0 \Let \sum_{i=1}^{r} w^{0}_i \mathcal{N}(\mu_0^i,\Sigma_0^i)$ to a target Gaussian mixture $\rho_{\rm f} \Let \sum_{j=1}^{s} w^{\rm f}_j \mathcal{N}(\mu_{\rm f}^i,\Sigma_{\rm f}^j)$. 
Each component pair $(i,j)$ is driven by its own affine controller, and assignment is randomized through an optimal transport coupling $\pi$ (see, e.g., \cite{villani_ot}), namely, a joint distribution whose marginals are the mixture weights $w^0$ and $w^{\rm f}$, where
\begin{align*}
    \sum_j \pi_{ij} = w^0_i, \quad \sum_i \pi_{ij} = w^{\rm f}_j, \quad \pi_{ij}\ge 0.
\end{align*}
We introduce $V_{ij}(\tf)$ to denote the value function of \Cref{prob:Cov:Steer} for the pair $(i,j)$, that is, the optimal steering cost between $\Sigma_0^i$ and $\Sigma_{\rm f}^j$ over horizon $\tf$, including the time penalty ${\varpi}\tf/2$. 
Solving this \emph{synchronized} problem reduces to evaluating each $V_{ij}(\tf)$ with the \texttt{CS} module and minimizing their total cost over the horizon,
\begin{align*}
C_{\rm sync} \Let \min_{\tf>0} \min_{\pi} \sum_{i,j}\pi_{ij} V_{ij}(\tf),
\end{align*}
which is essentially a one-dimensional line-search problem over $\tf$ (i.e., amounting to the transversality condition $\sum_{ij}\pi_{ij} {\hamil}_{ij}(\tf)=0$).
In contrast, our free-final time formulation lets each pair have its own optimal horizon ${\tf}^{ij}$ via its individual transversality condition ${\hamil}_{ij}=0$ using \Cref{alg_1}, yielding
\begin{align*}
C_{\rm ff} &\Let \min_{\pi}\ \sum_{i,j}\pi_{ij} \min_{\tf>0}V_{ij}(\tf) = \min_{\pi}\ \sum_{i,j}\pi_{ij}\,V_{ij}({\tf}^{ij}).
\end{align*}
The common final time strategy can be restrictive for general distribution steering problems, as it is not necessarily the minimizing solution for the decoupled problem, that is, $C_{\rm sync}\ge C_{\rm ff}$. 
We call this gap $\Delta \Let C_{\rm sync}-C_{\rm ff}\ge 0$ the \emph{price of synchronization}.

Figure~\ref{fig:gmm_price} depicts this gap by enforcing a common arrival time, showing an increased control effort cost of $14\%$, while the gap shrinks monotonically as the components merge, where the responsibility\footnote{For a sampled random variable $x$, 
the posterior probability that it belongs to component $i$ of a GMM
is known as the
\emph{responsibility} \cite[Equation 2.192, Section 9.2]{bishop2006prml}.}
entropy $\bar{\gamma}$ defined as
\begin{align*}
    \gamma_i(x) &\Let w^0_i\mathcal{N}(x;\mu^i_0,\Sigma^i_0) / \sum_k w^0_k\mathcal{N}(x;\mu^k_0,\Sigma^k_0), \\
    \bar{\gamma} &\Let \mathbb{E}_{x\sim\rho_0} \Big[-\sum_i \gamma_i(x)\log\gamma_i(x)\Big],
\end{align*}
grows, and randomized assignment becomes necessary. 
Figure \ref{fig:gmm_density} contrasts the two policies: the synchronized policy steers every component together at $\tf$, whereas the free-final time policy produces distinct arrivals ${\tf}^{ij}$, each held upon arrival.

\begin{figure*}[t]
\centering
\includegraphics[scale=0.37]{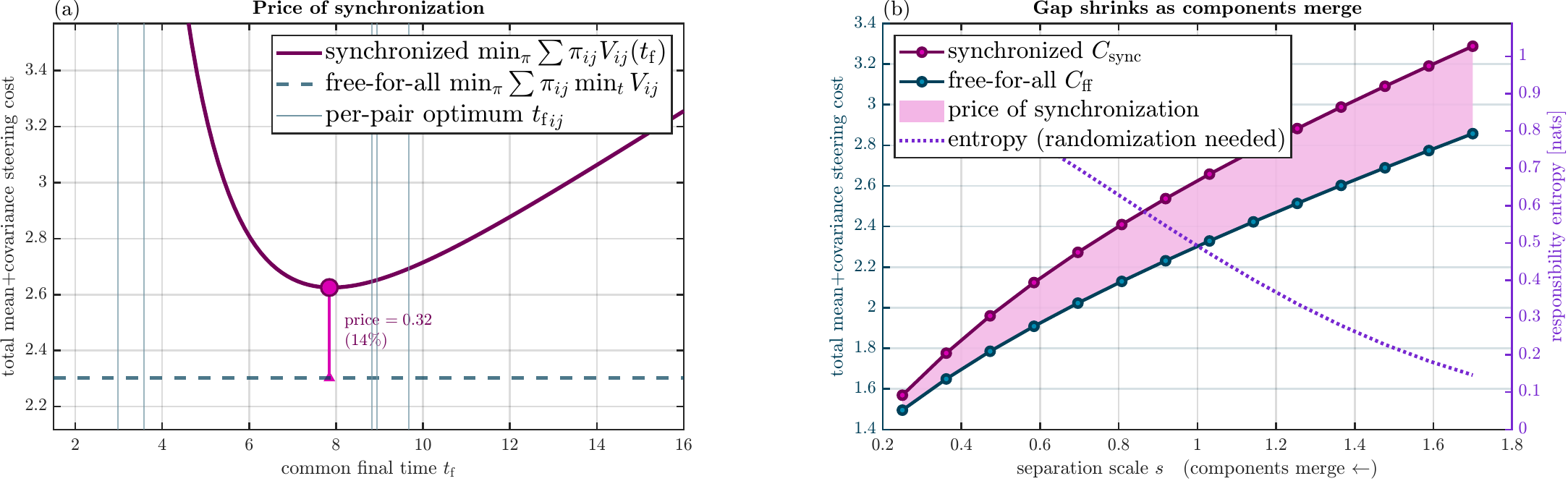}
\caption{Price of synchronization. (a) Total cost versus the common final time $\tf$ (synchronized) against the free-final time cost $C_{\rm ff}$. (b) The gap $\Delta$ and the responsibility entropy as the mixture components merge.}
\label{fig:gmm_price}
\end{figure*}

\begin{figure*}[t]
\centering
\includegraphics[scale=0.37]{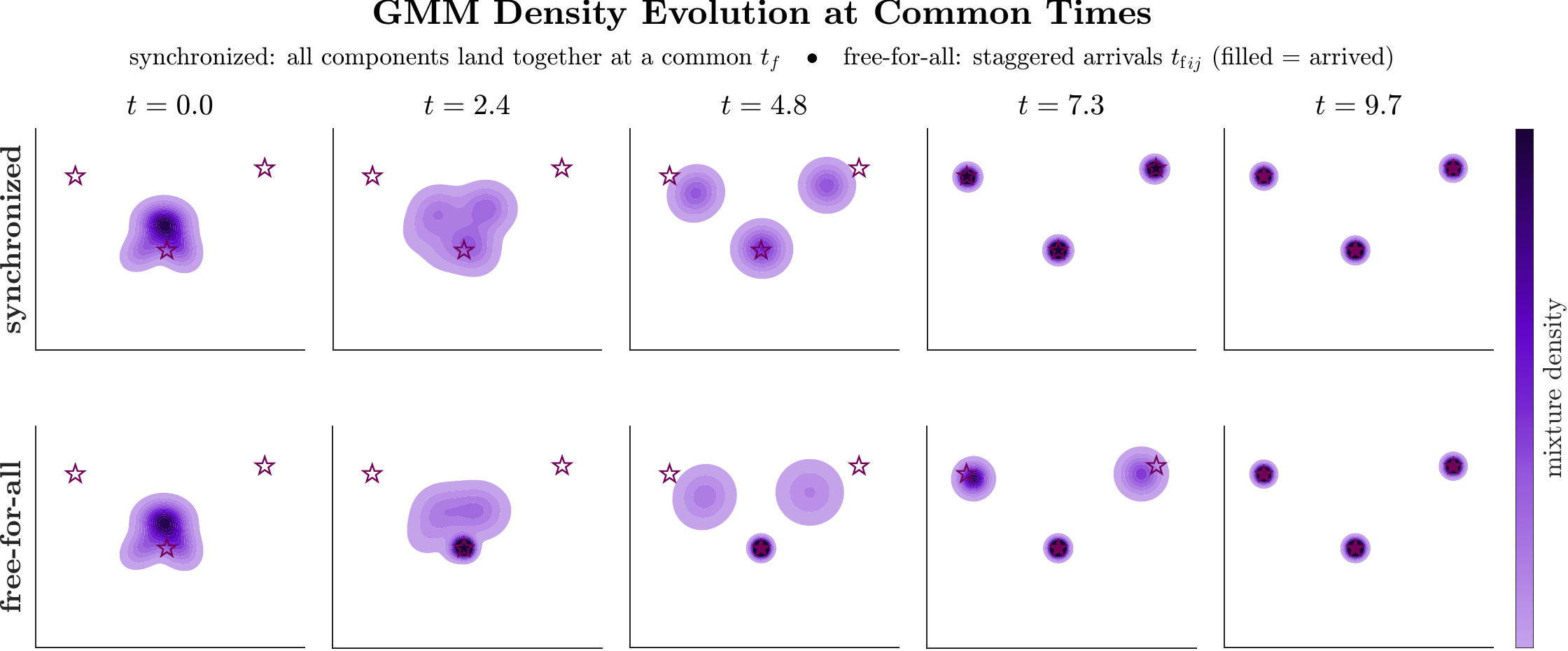}
\caption{Mixture-density evolution at common times. Top: the synchronized policy steers all components to their slots together at $\tf$. Bottom: the free-final time policy yields distinct arrivals ${\tf}^{ij}$.}
\label{fig:gmm_density}
\end{figure*}

%% file: sections/Conclusion.tex
\section{Conclusion} \label{sec:conc}

This article introduces the first systematic study of the continuous-time covariance steering problem with a common free-final time, in which the final time is free and optimized jointly with the covariance steering policy rather than imposed a priori.
We have derived the associated transversality conditions and characterized when Hamiltonian zeros are finite or asymptotic by introducing the notion of Asymptotic Gateways and Connections.
We have also established a trust-region line-search algorithm that detects both finite optimal times and infinite-horizon solution regimes. 
We have demonstrated the architecture's capabilities through three numerical examples.

Apart from its applicability to several practical problems, the proposed approach offers new theoretical insights into when free-final-time stochastic control problems may lead to asymptotic Hamiltonian zeros or to singular, and even unbounded, covariance behavior.
Second, our formulation provides a natural mechanism for more general free final-time distribution-steering problems: for instance, different Gaussian components of a GMM can be assigned their own, separately optimized final times rather than being stopped at a common time, thereby avoiding the cost of synchronization. 

Future work includes studying infinite-horizon continuous-time covariance steering under weaker hypotheses, for instance, by relaxing both the controllability assumption Assumption~\ref{assump:standing:1}
and Assumption~\ref{assump:point_limit} on the existence of a fixed asymptotic limit for the attained covariance.

%% file: sections/appendixA.tex
\section{}\label{sec:appen:a}

\begin{proof}[\highlight{Proof of Proposition \ref{Prop:IV1}}]
The result follows from a completion-of-squares argument \cite{cite3_1}.
Specifically, fix an admissible final time \(\tf>0\). 
Let 
$
[0,\tf] \ni t \mapsto S(t) \in \mathbb{S}^n 
$
be a continuously differentiable symmetric matrix-valued trajectory, to be chosen later. 
Consider the class of admissible controls on \([0,\tf]\), whose induced state process satisfies the prescribed covariance boundary conditions \(x(0)\sim\mathcal{N}(0,\initcov)\) and \(x(\tf)\sim\mathcal{N}(0,\fincov)\). 
For every control in this fixed-horizon feasible class, the quadratic terms  \(\mathbb E[x(\tf)\t S(\tf)x(\tf)]=\trace(S(\tf)\fincov)\) and \(\E[x(0)\t S(0)x(0)]=\trace(S(0)\initcov)\) are constant.

We define the performance index, $\mathcal{J}(\cdot) =\mathcal{J}(\tf,u(\cdot))$, with boundary terms that are constant under the prescribed covariance constraints, by
\begin{align}
    \hspace{-1 em} \mathcal{J}(\cdot) &\Let J(\cdot)+\half \E \left\{ x (\tf) \t S(\tf) x(\tf) -x(0)\t S(0) x(0) \right\} \nonumber \\
    & = J(\cdot)+\half \trace \left(S(\tf) \fincov \right)-\half\trace\left(S(0) \initcov\right)\nonumber \\
    & = J(\cdot)+\half \E \left \{ \int_0^{\tf} \, \d(x (t) \t S(t) x(t)) \right \} \label{mod-obj}.
\end{align}
Since the minimization of $\mathcal{J}(\cdot)$ in \eqref{mod-obj} is equivalent to the minimization of  $J(\cdot)$ in \eqref{eq:cost}, applying It\^o's Lemma \cite{book1} to  $\d(x (t) \t S(t) x(t))$ in \eqref{mod-obj} yields
\begin{align} \label{mod-obj-aux}
    \half \d( x(t)\t S(t) x(t)) &= \Big[ \half x(t) \t  (\dot S(t)+A\t S(t)+S(t) A)x(t) \Big. \nonumber \\
    &\Big. \underbrace{+u(t) \t B\t S(t) x(t)+\half \trace(S(t) D) \Big] \, \d t}_{\rm \text{drift term}}   +\underbrace{x(t)\t S(t) \noise \, \d w(t)}_{\rm \text{diffusion (martingale) term}}.
\end{align}
Taking the expectation in \eqref{mod-obj-aux}, and using the fact that the martingale term has a zero expectation, only the expectation of the drift term remains. 
Substituting \eqref{mod-obj-aux} into \eqref{mod-obj} yields
\begin{align} \label{mod-obj-2}
&\mathcal{J}(\cdot) =\E \Bigg \{ \int_0^{\tf} \, \Big[\half \varpi+\half u(t)\t R u(t)+\half x(t) \t  (Q+\dot S(t) \nonumber \Big. \Bigg. 
\\
    & \Bigg. \Big. \hspace{4 em}+A\t S(t)+S(t) A)x(t) + u(t) \t B\t S(t) x(t) + \half \trace(S(t) D) \Big] \d t \Bigg \}.
\end{align}
Choosing \(S(\cdot)\) to satisfy the matrix-valued Riccati equation
\begin{align} \label{ricc_S_dyn}
    \dot S(t)=-Q-A\t S(t)-S(t)A+S(t)B R^{-1} B\t S(t),
\end{align}
for all \(t \in [0,\tf]\) with fixed boundary conditions at $t=0$ and $t=\tf$ chosen such that the induced optimal feedback control steers the covariance from $\Sigma(0)=\initcov$ to $\Sigma(\tf)=\fincov$, and
substituting \eqref{ricc_S_dyn} into \eqref{mod-obj-2} yields,
\begin{align}
    \mathcal{J}(\cdot) \nonumber 
    &=\E \Bigg \{ \int_0^{\tf} \, \Big[\half \varpi+ \half u(t)\t R u(t) + u(t) \t B\t S(t) x(t) \nonumber  \\
    &\hspace{1 em} +\half \| B\t S(t) x(t) \|^2_{R^{-1}} 
    +\half \trace(S(t) D) \Big] \d t \Bigg \}, \nonumber \\
    &=\E \Bigg \{ \int_0^{\tf} \, \Big [\half 
    \norm{u(t)+R^{-1}B\t S(t)x(t)}_R^2  +\half \varpi +\half \trace(S(t) D) \Big] \, \d t \Bigg \} \label{mod-obj-moc},
\end{align}
where \(v \mapsto \norm{v}_R^2\Let v\t Rv\). Since \(R\succ 0\), the square term is minimized pointwise by \(u\as(t)=-R^{-1}B\t S(t)x(t)\) for all \(t\in[0,\tf].\) 
Hence, the optimal controller (whenever it exists) has the linear state-feedback form \(t \mapsto u^*(t)=K\as(t)x(t)\), with \(t \mapsto K\as(t) \Let -R^{-1}B\t S\as(t)\). The function \(S\as(\cdot)\) associated with the optimizer is selected together with the covariance trajectory so that the conditions \(\Sigma(0)=\initcov\) and \(\Sigma(\tf\as)=\fincov\) are satisfied. This proves our claim.
\end{proof}

\begin{proof}[\highlight{Proof of Theorem \ref{Theorem1}}]
A direct application of the PMP~\cite{ref:VinBet-DynOpt} to \Cref{prob:Det:Cov:Steer} yields that the optimal feedback gain is,
\begin{align} \label{theorem1:opt_cont}
   t \mapsto K^*(t)=-R^{-1} B\t \Pi(t),
\end{align}
where the corresponding closed-loop state equation is given by
\begin{align}
\hspace{-1 em} \dot{\Sigma}(t)&=A\Sigma(t)+\Sigma(t)A\t-BR^{-1}B\t\Pi(t)\Sigma(t)\nn \\& -\Sigma(t)\Pi(t)BR^{-1}B\t+D \, \text{ for all }t\in[0,\tfc], \label{final_cov_dyn}
\end{align}
with $\Sigma(0)=\initcov $, $\Sigma(\tfc) = \fincov$,
and the co-state matrix $\Pi(\cdot)$ satisfies~\eqref{final_ricc_dyn} with both $\Pi(0)$, $\Pi(\tfc)$ being free. 
\end{proof}

%% file: sections/appendixB.tex
\section{}\label{sec:appen:b}

This section establishes the statements and proofs of the technical results, mainly the continuity, boundedness, and compactness properties of the co-state/covariance trajectories, and the asymptotic behavior of the Hamiltonian. These results are used to support the analysis of Section~\ref{Section_5}.

\begin{myOCP}
\begin{lemma}[Continuity of the optimal co-state and the Hamiltonian]
\label{lem:reg:init_cost}
Consider the family of optimal covariance/co-state trajectories
$\{(\Sigma_{\tf}(\cdot),\Pi_{\tf}(\cdot))\}_{\tf>0}$ associated with the fixed-final time version of \Cref{prob:Det:Cov:Steer}, and let Assumption~\ref{assump:standing} hold.
Then, the maps $\tf\mapsto\Pi_{\tf}(0)$ and $(\tf,t)\mapsto\big(\Sigma_{\tf}(t),\Pi_{\tf}(t)\big)$, for $0\le t\le\tf$, are continuous on $(0,+\infty)$.
Consequently, the map $\tf\mapsto\hamil_{\tf}(\cdot)$ is continuous on $(0,+\infty)$.
\end{lemma}
\end{myOCP}

\begin{proof}
Under Assumption \ref{assump:standing:1}, for every fixed finite \(\tf>0\), \Cref{prob:Det:Cov:Steer} admits a unique covariance/co-state solution \(t \mapsto \bigl(\Sigma_{\tf}(t),\Pi_{\tf}(t)\bigr)\) for any prescribed endpoint covariances \(\initcov,\fincov\in\mathbb{S}^n_{++}\) \cite[Theorem 3]{cite17}. 
We prove the asserted continuity by viewing the initial co-state $\Pi_{\tf}(0)$
as the solution of a \emph{parameter-dependent endpoint equation}.

Fix an arbitrary reference horizon \(\tf^0>0\), and define \(G\Let BR^{-1}B\t.\)
For \(\tau>0\) and \(P\in\mathbb{S}^n\), introduce the normalized time variable \(s\Let t/\tau\in[0,1]\), and consider the initial value problem
\begin{align}
&\frac{\dd}{\dd s}\widehat{\Pi}(s;\tau,P) =
\tau\Bigl(-Q-A\t\widehat{\Pi}(s;\tau,P)-\widehat{\Pi}(s;\tau,P)A \nn \\& \qquad\qquad+\widehat{\Pi}(s;\tau,P)G\widehat{\Pi}(s;\tau,P)\Bigr), \label{eq:normalized:riccati}\\
&\frac{\dd}{\dd s}\widehat{\Sigma}(s;\tau,P)
=
\tau\Bigl(\bigl(A-G\widehat{\Pi}(s;\tau,P)\bigr)\widehat{\Sigma}(s;\tau,P)
 \nn\\&\qquad\qquad + \widehat{\Sigma}(s;\tau,P)\bigl(A-G\widehat{\Pi}(s;\tau,P)\bigr)\t+D\Bigr), \text{ with} \label{eq:normalized:covariance}\\
&\widehat{\Sigma}(0;\tau,P)=\initcov\quad\text{and}\quad \widehat{\Pi}(0;\tau,P)=P. \label{eq:normalized:initial}
\end{align}
Let \(\mathcal{D}\subset(0,+\infty)\times\mathbb{S}^n\) denote the set of pairs \((\tau,P)\) for which the solution of \eqref{eq:normalized:riccati}--\eqref{eq:normalized:initial} exists on the entire interval \([0,1]\). Since the right-hand sides of \eqref{eq:normalized:riccati} and \eqref{eq:normalized:covariance} are smooth in \((\widehat{\Sigma},\widehat{\Pi},\tau)\), the set \(\mathcal{D}\) is open and the solution map
\begin{align*}
  [0,1] \times \mathcal{D} \ni  (s,\tau,P)\mapsto\bigl(\widehat{\Sigma}(s;\tau,P),\widehat{\Pi}(s;\tau,P)\bigr)
\end{align*}
is continuously differentiable on its domain; see, for example, \cite[Theorem~2.11]{teschl2012ode}.

Define now the endpoint residual map
\begin{align*}
\residue:\mathcal{D}\to\mathbb{S}^n,\quad
(\tau,P) \mapsto \residue(\tau,P)\Let\widehat{\Sigma}(1;\tau,P)-\fincov.
\end{align*}
For each fixed \(\tau>0\), the initial co-state \(P_{\tau}\Let\Pi_{\tau}(0)\) associated with the unique solution of the fixed-horizon problem is characterized by
\(\residue(\tau,P_{\tau})=0.\)
In particular, setting \(P^0\Let\Pi_{\tf^0}(0)\), we have \(\residue \bigl(\tf^0,P^0 \bigr)=0.\)

We next verify that the derivative of \(\residue\) with respect to \(P\) is nonsingular at \((\tf^0,P^0)\).
For the fixed parameter \(\tau=\tf^0\), the normalized equations \eqref{eq:normalized:riccati}--\eqref{eq:normalized:covariance} are the fixed-horizon covariance/co-state equations on \([0,1]\) associated with the scaled matrices
\[
A^0\Let\tf^0A, \,\,B^0\Let\sqrt{\tf^0}\,B,\,\, Q^0\Let\tf^0Q,\,\, \text{and } D^0\Let\tf^0 D.
\]
Since \((A,B)\) is controllable, the pair \((A^0,B^0)\) is also controllable. Therefore, the endpoint-map analysis in \cite[Lemma 8]{cite17} applies to the map \(P \mapsto\widehat{\Sigma}(1;\tf^0,P).\)
More precisely, the proof of \cite[Lemma 8]{cite17} shows that this map is continuously differentiable and that its Jacobian with respect to the initial co-state is nonsingular at every admissible initial co-state. Consequently,
\begin{align*}
\mathsf{D}_P\residue(\tf^0,P^0):\mathbb{S}^n \rightarrow \mathbb{S}^n
\end{align*}
is an isomorphism.\footnote{Here \(\mathsf{D}_P \residue\) is the usual (partial) Frechet derivative \cite[Chapter 5]{ref:DonRoc-14}.}
The implicit function theorem \cite[Section 1.2]{ref:DonRoc-14} now yields an open neighborhood \(\mathcal{O}_1\subset(0,+\infty)\) of \(\tf^0\), an open neighborhood \(\mathcal{O}_2\subset\mathbb{S}^n\) of \(P^0\), and a unique continuously differentiable map \(P:\mathcal{O}_1 \to \mathcal{O}_2\)
such that
$P(\tf^0)=P^0$ and
$\residue(\tau,P(\tau))=0$ for every $\tau\in\mathcal{O}_1$.
For each \(\tau\in\mathcal{O}_1\), both \(P(\tau)\) and \(\Pi_{\tau}(0)\) generate a solution satisfying the same covariance boundary conditions \(\widehat{\Sigma}(0;\tau,P)=\initcov\) and \(\widehat{\Sigma}(1;\tau,P)=\fincov.\)
The uniqueness of the fixed-horizon covariance/co-state solution \cite[Theorem 3]{cite17} therefore implies that \(P(\tau)=\Pi_{\tau}(0)\) for every \(\tau\in\mathcal{O}_1.\)
Hence, \(\tau\mapsto\Pi_{\tau}(0)\) is continuously differentiable, and, in particular, continuous at \(\tf^0\). 
Since \(\tf^0>0\) is arbitrary, the map \(\tf\mapsto\Pi_{\tf}(0)\) is continuous on \((0,+\infty)\).

We next establish continuity of the complete covariance/co-state trajectories. 
To this end, define the triangular domain
\begin{align*}
\mathcal{T}_D\Let \aset[\big]{(\tf,t)\in(0,+\infty)\times[0,+\infty) \suchthat 0\leq t\leq\tf }.    
\end{align*}

For every \((\tf,t)\in\mathcal{T}_D\), the original and normalized trajectories are related by
\begin{align*}
&\Sigma_{\tf}(t)=\widehat{\Sigma}\bigl(t/\tf;\tf,\Pi_{\tf}(0)\bigr), \\
&\Pi_{\tf}(t)=\widehat{\Pi}\bigl(t/\tf;\tf,\Pi_{\tf}(0)\bigr).
\end{align*}
The map \((\tf,t)\mapsto t/\tf\) is continuous on \(\mathcal{T}_D\), the map \(\tf\mapsto\Pi_{\tf}(0)\) is continuous, and the normalized solution map is continuous in time, the parameter, and the initial data. It follows by composition that the map
\(
    (\tf,t)\mapsto\bigl(\Sigma_{\tf}(t),\Pi_{\tf}(t)\bigr)
\)
is continuous on \(\mathcal{T}_D\).
Finally, and since the dynamics and the running cost are autonomous, the Hamiltonian is constant along every extremal. 
Thus, \(\hamil_{\tf}(t)=\hamil\bigl(\initcov,\Pi_{\tf}(0)\bigr)\) for all $t \in [0,\tf]$.
Since \((\Sigma,\Pi)\mapsto\hamil(\Sigma,\Pi)\) is continuous and \(\tf\mapsto\Pi_{\tf}(0)\) is continuous, the map \(\tf\mapsto\hamil_{\tf}(\cdot)\) is continuous on \((0,+\infty)\). Our proof is complete.
\end{proof}

We present a compactness property of the fixed-final time optimal covariance/co-state trajectories, which are the solutions to the fixed-final-time version of \Cref{prob:Det:Cov:Steer}. 
This property will be used in the subsequent large-horizon analysis to carry out existence arguments on compact time intervals. 

\begin{myOCP}
\begin{proposition}[Compactness of optimal covariance/co-state trajectories]
\label{cor:cov_costate_compactness}
Consider the family of optimal covariance/co-state trajectories
$\{(\Sigma_{\tf}(\cdot),\Pi_{\tf}(\cdot))\}_{\tf>0}$ associated with the fixed-final time version of \Cref{prob:Det:Cov:Steer}. Then: 
\begin{enumerate}[label={\textup{(\ref{cor:cov_costate_compactness}-\alph*)}}, leftmargin=*, widest=b, align=left]
\item \label{cor:cov_costate_compactness:1} Under Assumption~\ref{assump:standing}, the family $\{(\Sigma_{\tf}(\cdot),\Pi_{\tf}(\cdot))\}_{\tf>0}$ is uniformly bounded on compact subintervals of $\tf \in (0,+\infty)$. 
\item \label{cor:cov_costate_compactness:2} If, in addition, Assumption~\ref{assump:point_limit} holds, then the family is locally bounded in time, uniformly in $\tf$, that is, for each fixed $T>0$, there exists $\alpha_T> 0$ such that $\|\Sigma_{\tf}(t)\|+\|\Pi_{\tf}(t)\|\le\alpha_T$ for all $\tf \ge T$ and $t\in [0,T]$, i.e., 
\begin{align}
\hspace{-1 em} \sup_{\tf\ge T}\ \sup_{t\in[0,T]}\big(\|\Sigma_{\tf}(t)\|+\|\Pi_{\tf}(t)\|\big) \le \alpha_T. \label{eq:horizon_uniform_bound}
\end{align}
Hence, every subsequential limiting trajectory obtained as $\tf\to+\infty$ is bounded on compact intervals.
\end{enumerate}
\end{proposition}
\end{myOCP}

\begin{proof} Our proof consists of two steps.

\proofstep{coroll:step:one}{Uniform boundedness on compact subintervals} We fix a compact subinterval $[\tau_1,\tau_2]\subset(0,+\infty)$ with $0<\tau_1<\tau_2<+\infty$.
From \cite[Theorem 3]{cite17}, for every $\tf\in[\tau_1,\tau_2]$ the optimality system \eqref{final_ricc_dyn},~\eqref{final_cov_dyn} admits a unique solution $(\Sigma_{\tf}(\cdot),\Pi_{\tf}(\cdot))$.
Moreover, continuity of $(\Sigma_{\tf}(\cdot),\Pi_{\tf}(\cdot))$ on the compact set $[\tau_1,\tau_2]$ yields $m \Let \sup_{\tf\in[\tau_1,\tau_2]}\|\Pi_{\tf}(0)\| < +\infty$.
The right-hand sides of \eqref{final_ricc_dyn},~\eqref{final_cov_dyn} are locally Lipschitz in $(\Sigma,\Pi)$, so the solution of the associated ODE depends continuously on the initial data and on time \cite[Theorem~2.11]{teschl2012ode}. 
Restricting the initial data to the compact set 
\begin{align*}
    \mathcal{C}\Let\{\initcov\}\times\aset[]{\Pi_0\in\mathbb{S}^n \suchthat \|\Pi_0\|\le m},
\end{align*}
and the pair $(\tf,t)$ to the compact triangle $\mathcal{T}\Let\aset[]{(\tf,t)\suchthat \tf\in[\tau_1,\tau_2],\ 0\le t\le\tf}$, along which the optimal pairs exist by Fact~\ref{rem:fixed_horizon_wellposedness}, the map $(\tf,t)\mapsto(\Sigma_{\tf}(t),\Pi_{\tf}(t))$, being the composition of the continuous map $\tf\mapsto\Pi_{\tf}(0)\in\mathcal{C}$ with the solution map, is continuous on $\mathcal{T}$, and therefore its image is bounded.
Consequently,
\begin{align*}
\sup_{\tf\in[\tau_1,\tau_2]}\ \sup_{t\in[0,\tf]}\big(\|\Sigma_{\tf}(t)\|+\|\Pi_{\tf}(t)\|\big) < +\infty,
\end{align*}
which is the claimed uniform boundedness on $[\tau_1,\tau_2]$. 

\proofstep{coroll:step:two}{Boundedness of the limiting trajectories} Now suppose Assumption~\ref{assump:point_limit} also holds. 
Fix $\tau_3>0$. 
By Lemma~\ref{lem:reg:init_cost}, the map $\tf\mapsto\Sigma_{\tf}(\tau_3)$ is continuous on $[\tau_3,+\infty)$ and, by Assumption~\ref{assump:point_limit}, converges to $\Sigma_{\star}(\tau_3)$ as $\tf\to+\infty$. 
Hence, 
\begin{equation} \label{eq:SigmaB}
\sup_{\tf\ge\tau_3}\|\Sigma_{\tf}(\tau_3)\|<+\infty.
\end{equation}
We claim that
$
m_\infty \Let \sup_{\tf\ge \tau_3}\|\Pi_{\tf}(0)\| < +\infty.
$
To see this, assume, by contradiction, that no such bound exists. 
Then, there exists a sequence $\tf^k\ge\tau_3$ such that $\|\Pi_{\tf^k}(0)\|\to+\infty$. 
From Fact~\ref{rem:fixed_horizon_wellposedness} applied with the final time $\tau_3$, the map $\Pi(0)\mapsto\Sigma(\tau_3)$ is a homeomorphism, and therefore $\Phi_{12}(\tau_3,0)$ is invertible. 
Since $\Phi_{11}(\tau_3,0)$ and $\Phi_{12}(\tau_3,0)$ are fixed matrices that depend only on $\tau_3$, the only quantity varying along the sequence $\{\tf^k\}_{k\ge 1}$ is $\Pi_{t_{\mathrm f}^k}(0)$. 
Then, it follows from~\eqref{STM_Cov} that
\begin{align*}
&\|\bar\Phi_{\bar A}(\tau_3,0)\| 
= \|\Phi_{11}(\tau_3,0)+\Phi_{12}(\tau_3,0)\Pi_{\tf^k}(0)\| \\
&\hspace{1 em}\ge \sigma_{\min}\!\big(\Phi_{12}(\tau_3,0)\big)\,\|\Pi_{\tf^k}(0)\| - \|\Phi_{11}(\tau_3,0)\| \xrightarrow{k\to+\infty} +\infty.
\end{align*}
Since the diffusion integral term in the covariance solution~\eqref{cov_analytic_stm} is positive semidefinite and $\initcov\succ0$, we have
\begin{align*}
    \Sigma_{\tf^k}(\tau_3)&\succcurlyeq\bar\Phi_{\bar A}(\tau_3,0)\,\initcov\,\bar\Phi_{\bar A}(\tau_3,0)\t \succcurlyeq\lambda_{\min}(\initcov)\,\bar\Phi_{\bar A}(\tau_3,0)\bar\Phi_{\bar A}(\tau_3,0)\t,
\end{align*}
and
\begin{align*}
\|\Sigma_{\tf^k}(\tau_3)\| \ge \lambda_{\min}(\initcov)\,\|\bar\Phi_{\bar A}(\tau_3,0)\|^2 \xrightarrow{k\to+\infty} +\infty,
\end{align*}
contradicting the bound \eqref{eq:SigmaB}. 
Therefore, $m_\infty<+\infty$. 

Finally, fix any $T\ge\tau_3$ and restrict the initial data to the compact set $\mathcal{C}_\infty\Let\{\initcov\}\times\aset[]{\Pi_0\in\mathbb{S}^n \suchthat \|\Pi_0\|\le m_\infty}$. 
For every $\tf\ge T$, the optimal pair 
$(\Sigma_{\tf}(\cdot),\Pi_{\tf}(\cdot))$ 
exists on $[0,\tf]\supseteq[0,T]$ by Fact~\ref{rem:fixed_horizon_wellposedness}, with initial data $(\initcov,\Pi_{\tf}(0))\in\mathcal{C}_\infty$.
Hence, following the same argument as in \ref{coroll:step:one}, the solution map 
$(\tf,t)\mapsto(\Sigma_{\tf}(t),\Pi_{\tf}(t))$
is continuous on the compact set $\mathcal{C}_\infty\times[0,T]$.
It follows that its image is bounded and hence there exists some $\alpha_T > 0$ such that \eqref{eq:horizon_uniform_bound} holds. 
Taking limits yields boundedness of every subsequential limiting trajectory on compact intervals. 
The assertion stands established.
\color{black}
\end{proof}

\begin{proof}[\highlight{Proof of Theorem~\ref{prop_exist_AC}}]
Our proof consists of five steps.

\noindent\proofstep{step:one}{Compactness and extraction of limiting arcs}
Let $\{\tf^k\}_{k\ge 1}\subset(0,+\infty)$ be any sequence such that $\tf^k\to+\infty$, and let $\bigl(\Sigma_k(\cdot),\Pi_k(\cdot)\bigr) \Let \bigl(\Sigma_{\tf^k}(\cdot),\Pi_{\tf^k}(\cdot)\bigr)$ denote the corresponding optimal covariance/co-state trajectories on $[0,\tf^k]$.
Fix $T>0$. Since $\tf^k\to+\infty$, one has $[0,T]\subset[0,\tf^k]$ for all sufficiently large $k$. 

By Assumptions~\ref{assump:standing:1}, ~\ref{assump:point_limit}, and the uniform boundedness result of Proposition~\ref{cor:cov_costate_compactness}, the family $\{(\Sigma_k,\Pi_k)\}_{k\ge1}$ is uniformly bounded on $[0,T]$ uniformly in $k$. 
Hence, there exists $C_T>0$ such that, for all sufficiently large $k$,
\begin{align*}
\sup_{t\in[0,T]} \bigl( \|\Sigma_k(t)\|+\|\Pi_k(t)\|\bigr)\le C_T .
\end{align*}
Then, from the Riccati and covariance dynamics \eqref{final_ricc_dyn} and \eqref{final_cov_dyn}, we obtain, for all sufficiently large $k$ and all $t\in[0,T]$,
\begin{align*}
    \|\dot\Pi_k(t)\|  &\le \ell_{\Pi,T}\Let\|Q\|+2\|A\|C_T+\|BR^{-1}B\t\|C_T^2, \\ \|\dot\Sigma_k(t)\|
    &\le \ell_{\Sigma,T} \Let2\|A\|C_T +2\|BR^{-1}B\t\|C_T^2 +\|D\|,
\end{align*}
It follows that the family $\aset[\big]{(\dot\Sigma_k,\dot\Pi_k)}_{k\ge1}$ is uniformly bounded on $[0,T]$.
Moreover, by defining $\ell_T \Let \max\{\ell_{\Pi,T},\ell_{\Sigma,T}\}$, we see that, for all sufficiently large $k$ and all $t\in[0,T]$,
$
    \|\dot\Pi_k(t)\| \le \ell_T \,\, \text{and}\,\,  \|\dot\Sigma_k(t)\|  \le \ell_T.
$
Consequently, for any $s,t\in[0,T]$ and for sufficiently large $k$,
\begin{align*}
    \norm{\Pi_k(t)-\Pi_k(s)} &\le \int_s^t \|\dot\Pi_k(\tau)\|\,\d \tau \le  \ell_T |t-s|,\\ 
    \norm{\Sigma_k(t)-\Sigma_k(s)} &\le  \int_s^t \|\dot\Sigma_k(\tau)\|\,\d \tau  \le \ell_T |t-s|.
\end{align*}
Thus, from~\cite[Section 3.1.10]{ref:LipFunc:book}, the family $\aset[\big]{(\Sigma_k,\Pi_k)}_{k\ge1}$ is equi-Lipschitz on $[0,T]$ and hence equicontinuous on $[0,T]$.
By the Arzel\`{a}--Ascoli theorem \cite[Theorem 7.25]{rudin1976}, one may extract a subsequence converging uniformly on $[0,T]$. 
Applying this argument for $T=1,2,\dots$ and using a standard diagonal extraction argument (see, e.g., \cite[Theorem 7.23]{rudin1976}) yields a subsequence, still denoted by $\aset[\big]{(\Sigma_k,\Pi_k)}_{k\ge1}$ for simplicity, and a limiting trajectory pair
$
\bigl(\Sigma^+(\cdot),\Pi^+(\cdot)\bigr):[0,+\infty)\to \mathbb{S}_{+}^n\times\mathbb{S}^n,
$
such that $(\Sigma_k,\Pi_k)\to(\Sigma^+,\Pi^+)$ as $k\to +\infty$, uniformly on every compact interval of $[0,+\infty)$.

Similarly, for all $\tau\in[0,\tf^k]$, define the time-reversed trajectories
\begin{align} \label{eqn:time_rev_cov_ricc}
\hat\Sigma_k(\tau) \Let \Sigma_k(\tf^k-\tau), \quad \hat\Pi_k(\tau) \Let \Pi_k(\tf^k-\tau).
\end{align}
Applying the same compactness argument yields a subsequence converging uniformly on compact intervals of $[0,+\infty)$ to a backward limiting trajectory $\bigl(\Sigma^-(\cdot),\Pi^-(\cdot)\bigr)$, 
\black{that is,
$(\hat\Sigma_k,\hat\Pi_k)\to(\Sigma^-,\Pi^-)$ as $k\to +\infty$, uniformly on every compact interval of $[0,+\infty)$, where each fixed compact interval is contained in $[0,\tf^k]$ for all sufficiently large $k$.}

\medskip

\noindent\proofstep{step:two}{Finiteness of the limiting running cost}
Since $\mathcal G_\infty\neq\varnothing$, there exist $\Sigma_{\rm s}\in\mathcal G_\infty$ such that $A\Sigma_{\rm s}+\Sigma_{\rm s}A\t+D=0$ and $\trace(Q\Sigma_{\rm s})=0$; 
see Definition~\ref{asympt_gateway_2}.
Moreover, $\Sigma_{\rm s}\succ0$ and $Q\succcurlyeq0$, along with the latter equality implies that $Q=0$.

By the feasibility of the fixed-final time covariance steering problem between two prescribed positive-definite covariance boundary conditions (under Assumption~\ref{assump:standing}, see Fact~\ref{rem:fixed_horizon_wellposedness}), there exist $t_1,t_2>0$ and admissible feedback gains that steer $\initcov $ to $\Sigma_{\rm s}$ on $[0,t_1]$ and $\Sigma_{\rm s}$ to $\fincov $ on $[0,t_2]$.
Hence, for any $\tf>t_1+t_2$, one may construct a feasible concatenated covariance trajectory by steering from $\initcov $ to $\Sigma_{\rm s}$ on $[0,t_1]$, then setting $K(\cdot)\equiv0$ on $[t_1,\tf-t_2]$, and finally steering from $\Sigma_{\rm s}$ to $\fincov $ on $[\tf-t_2,\tf]$.

Similarly, since $\Sigma_{\rm s}\in\mathcal G_\infty$, the uncontrolled covariance dynamics keeps $\Sigma_{\rm s}$ invariant on the time interval $[t_1,\tf-t_2]$. 
Moreover, since $\varpi=0$, $Q=0$, and $K(\cdot)\equiv0$ on that interval, it contributes a zero running cost.
Therefore, the total cost of the constructed feasible trajectory becomes independent of \(\tf\). Let this cost be denoted by some constant \(C\ge 0\), and 
let \(K_{\rm{fs}}^{\tf}(\cdot)\) denote the feedback gain associated with this concatenated feasible trajectory. 

Recall the fixed-final time value function $\tf \mapsto V(\tf)$ for \Cref{prob:Det:Cov:Steer} from Corollary \ref{cor_dv_H_0}
\begin{align}
    V(\tf) = \inf_{K\in \mathcal{K}}  J^{\rm P2}(\tf,K), \nn
\end{align}
subject to \(\Sigma(0)= \initcov\) and \(\Sigma(\tf) = \fincov\). Since \(V(\tf)\) is the infimum of \(J^{\rm P2}(\tf,K)\) over all feasible gains satisfying the endpoint constraints, we have \(V(\tf)\le J^{\rm P2}\bigl(\tf,K_{\rm fs}^{\tf}\bigr)=C\) for all sufficiently large \(\tf\).

Now let \(\tf^k\to+\infty\) as \(k \to +\infty\), be the final time sequence introduced in \ref{step:one}, and let \(\aset[\big]{(\Sigma_k,\Pi_k)}_{k\ge1}\) be the corresponding optimal covariance/co-state trajectory on \([0,\tf^k]\). 
Since \(\varpi=0\) and, as shown above, \(Q=0\), the fixed-final time optimal value reduces to the optimal control-energy cost. Using the optimal feedback \(K_k(t)=-R^{-1}B\t\Pi_k(t)\), define
\begin{align*}
t \mapsto g_k(t)\Let \trace\big(\Pi_k(t)\t BR^{-1}B\t\Pi_k(t)\Sigma_k(t)\big)\ge0,
\end{align*}
for all $t \in[0,\tf^k]$ and
\begin{align*}
   t\mapsto g^+(t)\Let \trace\big(\Pi^+(t)\t BR^{-1}B\t\Pi^+(t)\Sigma^+(t)\big) \ge 0,
\end{align*}
for all $t\ge 0$.
Then, $V(\tf^k)=\frac12\int_0^{\tf^k} g_k(t)\, \d t \le C$. For any fixed $T>0$, the uniform convergence $(\Sigma_k,\Pi_k)\to(\Sigma^+,\Pi^+)$ on $[0,T]$ implies that $g_k\to g^+$ uniformly on $[0,T]$. As a result,
\begin{align} \label{runn_cost_suff_large}
\int_0^T g^+(t)\, \d t
&=\lim_{k\to+\infty}\int_0^T g_k(t)\, \d t \le \limsup_{k\to+\infty}\int_0^{\tf^k} g_k(t)\, \d t \le 2C.
\end{align}

Since $g^+\ge0$, we take the limit $T\to+\infty$, to obtain
\begin{align*}
\int_0^{+\infty} g^+(t)\,\d t \le 2C<+\infty.
\end{align*}
Thus, the forward limiting arc has a finite infinite-horizon running cost. 
The same argument also holds for the backward limiting arc.

Finally, we show that the value function $V(\cdot)$ admits an alternative representation that will be useful in the following LaSalle analysis. 
Consider the optimal covariance/co-state pair $(\Sigma(\cdot),\Pi(\cdot))$ for a given fixed-final time $\tf>0$ with optimal control $K(t)=-R^{-1} B\t \Pi(t)$ for all $t\in[0,\tf]$. 
Then, from the Riccati and covariance dynamics \eqref{final_ricc_dyn} and \eqref{final_cov_dyn}, it follows that
\begin{align}\label{def:V_alt_pre_expr}
    \frac{\d}{\d t} & \trace(\Pi(t) \Sigma(t)) =-\trace(Q\Sigma(t))-\trace(\Pi(t)BR^{-1} B\t \Pi(t) \Sigma(t)) \nn\\
    &\hspace{1 em}+ \trace(\Pi(t) D),
\end{align}
for all $t\in[0,\tf]$. 
Therefore, by integrating both sides \eqref{def:V_alt_pre_expr} over the time interval $[0,\tf]$, and using the boundary conditions $\Sigma(0)=\initcov$ and $\Sigma(\tf)=\fincov$, yields
\begin{align}
\label{def:V_alt_expr}
    V(\tf)=\half \trace(\Pi_0 \initcov) 
    - \half \trace(\Pi_{\rm f} \fincov) + \half \int_0^{\tf} \trace(\Pi(t) D) \, \d t.
\end{align}

\medskip

\noindent\proofstep{step:three}{Forward and backward LaSalle arguments}
Define an accumulator function $\A:\mathbb{R} \to \mathbb{R}$ and let  $\dot {\A}(t)=\trace(\Pi(t) D)$ for all $t\ge 0$ with ${\A}(0)=0$. 
Take the negative of the alternative expression for the value function from \eqref{def:V_alt_expr} as the Lyapunov-like function 
\((\Sigma(t),\Pi(t),\A(t)) \mapsto \V(\Sigma(t),\Pi(t),\A(t)) \Let \trace(\Sigma(t)\Pi(t))-\A(t). 
\)
With a slight abuse of notation, let $\V(t) \Let \V(\Sigma(t),\Pi(t),\A(t))$.

We first show that the limiting co-state $\Pi^+(\cdot)$ is bounded on $[0,+\infty)$. 
By Assumption~\ref{assump:point_limit}, $\Sigma^+(t) \succcurlyeq \varepsilon I_n$ for some $\varepsilon>0$ and all sufficiently large $t$, so the finite-cost bound of \ref{step:two} yields 
\begin{align*}\int_0^{+\infty}\trace\big(\Pi^+(t) BR^{-1}B\t\Pi^+(t)\big)\,\d t \le \varepsilon^{-1} \int_0^{+\infty}g^+(t)\,\d t<+\infty.
\end{align*}
Together with the global definedness of the limiting arc, and the explicit solution of the Riccati flow in~\eqref{final_ricc_dyn} with $Q=0$, it follows that $\Pi^+(\cdot)$ remains bounded on $[0,+\infty)$ (e.g., see also \cite{riccati_inf} for convergence criteria for matrix Riccati differential equations).

Boundedness of ${\A}(\cdot)$ along the limiting arc follows from integrating \eqref{def:V_alt_pre_expr} with $Q=0$, which yields the identity
\begin{align*}
A(t)=\trace(\Pi(t)\Sigma(t))-\trace(\Pi(0)\Sigma(0))+\int_0^{t} g^+(s)\,\d s,
\end{align*}
where the right-hand side is bounded on $[0,+\infty)$ by Assumption~\ref{assump:point_limit} and $\int_0^{+\infty}g^+\le 2C$ from \ref{step:two}.

Moreover, by \eqref{eq:horizon_uniform_bound} of Proposition~\ref{cor:cov_costate_compactness}, the limiting arc $(\Sigma^+(\cdot),\Pi^+(\cdot))$ is uniformly bounded, and the dynamics \eqref{final_ricc_dyn}, \eqref{final_cov_dyn} together with $\dot\A(t)=\trace(\Pi(t) D)$ are autonomous, hence its $\omega$-limit set is nonempty, compact and invariant \cite[Lemma~4.1]{ref_khalil}.

Along the covariance dynamics~\eqref{final_ricc_dyn},~\eqref{final_cov_dyn}, taking the time derivative of $\V(t)$ leads to  
\begin{align*}
    \dot \V(t)  =-\trace(Q\Sigma(t)) - \trace(\Pi(t)\t B R^{-1} B\t \Pi(t) \Sigma(t)) \le 0.
\end{align*}
Consequently, $\dot \V(t) = 0$ is satisfied when $\trace(Q\Sigma)=0$ (equivalently, $Q=0$ since $\Sigma\succ 0$), and hence
\begin{align*}
    \trace(\Pi\t B R^{-1} B\t \Pi \Sigma)=0.
\end{align*}
Since $\trace(\Pi\t B R^{-1} B\t \Pi \Sigma)=\lVert R^{-\half}B\t \Pi \Sigma^{1/2} \rVert_{\rm F}^2\ge 0,$ it follows that $B\t \Pi \Sigma^{1/2}=B\t \Pi \Sigma=0$ and $\Sigma \succ 0$, which imply that $B\t \Pi \equiv0$. 
    Therefore, on the set $\{\dot \V=0 \}$, we have 
    \begin{equation}  \label{eqn:Sdot:unc}
    \dot \Sigma = A \Sigma + \Sigma A\t+D.
    \end{equation}
    Furthermore, by taking successive derivatives of $B\t \Pi$ with $Q=0$ and setting them to zero, we get the equality $\Pi A^k B=0$ for all $k\in \{0,1,\dots,n-1\}$. 
    Finally, by Assumption~\ref{assump:standing:1}, $(A,B)$ is controllable, and therefore $\Pi$ is necessarily $0$. 
    The question now is, given that the covariance solution for the fixed-final time version of \cref{prob:Det:Cov:Steer} 
    is bounded from Proposition~\ref{cor:cov_costate_compactness}, whether a bounded solution of  
    the uncontrolled covariance dynamics \eqref{eqn:Sdot:unc} converges to a stationary covariance (Assumption~\ref{assump:point_limit}), or whether it can remain bounded without converging. The solution to the uncontrolled covariance dynamics is
    \begin{align*}
       t \mapsto \Sigma(t)=e^{At}\initcov  e^{A\t t}+ \int_0^t e^{As} D e^{A\t s} \d s \quad\text{for all } t\ge 0,
    \end{align*}
    where, $\int_0^t e^{As} D e^{A\t s} \d s$ is the controllability Gramian of $(A,\Gamma)$. 
    Observe that the large-time behavior is determined by the spectrum of \(A\) and by how the diffusion directions encoded by \(D\) excite the modes of \(A\). 
    This leads to the following modal requirements for boundedness and positive definiteness of $t \mapsto \Sigma(t)$ for all $t>0$:
    \begin{enumerate}[label={\textup{(\(\textsf{M}\)-\alph*)}}, leftmargin=*, widest=b, align=left]
        \item \label{cond:1} Any purely imaginary eigenvalue of \(A\), if present, must be semisimple \cite[Def.~7.1]{hespanha2018linear}  and its associated eigenvectors must lie in \(\mathrm{ker} D\);
        \item \label{cond:2} The eigenvalues of \(A\) with negative real part must be controllable through the pair \(\bigl(A,\Gamma\bigr)\).
    \end{enumerate}
        
    The first condition prevents polynomial growth or diffusion-driven growth along marginally stable directions. 
    The second condition ensures that the stable part of the covariance does not lose rank in the directions relevant to the diffusion-driven stationary covariance.

    These modal requirements then lead to two qualitatively different bounded regimes.
    \begin{enumerate}[label={\textup{(\(\bar{\textsf{M}}\)-\alph*)}}, leftmargin=*, widest=b, align=left]
    \item  \label{cond:3} \(A\) is Hurwitz and \(\bigl(A,\Gamma\bigr)\) is controllable. In this case, the uncontrolled covariance dynamics converge to the unique positive-definite solution \(\Sigma_s\) of the Lyapunov equation
    \begin{align*} 
    A\Sigma_s+\Sigma_s A\t + D=0.
    \end{align*}
    
    Thus, \(\lim_{t\to+\infty}\Sigma(t)=\Sigma_s\), and consequently \(\lim_{t\to+\infty}\dot\Sigma(t)=0\). As a result, substituting $\dot \Sigma=0$, $Q\Sigma=0$ and $\Pi=0$ into \eqref{final_ham_dyn}, we have an asymptotic Hamiltonian-zero, i.e., $\hamil_{\tf}(\cdot)\equiv0$ as $\tf \to+\infty$.
    
    \item \label{cond:4} 
    All eigenvalues of \(A\) either have negative real part or lie on the imaginary axis; the imaginary-axis eigenvalues are semisimple; the associated marginal directions lie in \(\ker D\); 
    and the stable modes are controllable through \((A,\Gamma)\). In this case, the covariance trajectory remains bounded, but it need not converge to a single stationary covariance. Instead, oscillatory covariance behavior may occur, contradicting the Assumption \ref{assump:point_limit}.
\end{enumerate}

Under the second condition \ref{cond:4}, notice that the Lyapunov equation $A\Sigma + \Sigma A\t + D = 0$ has infinitely many positive-definite solutions for $\Sigma$. 
Similarly to the first condition \ref{cond:3}, substituting $Q\Sigma=0$ and $\Pi=0$ into \eqref{final_ham_dyn} yields asymptotic Hamiltonian-zero, regardless of Assumption \ref{assump:point_limit}.
    
In conclusion, the uncontrolled covariance dynamics~\eqref{eqn:Sdot:unc} has a \emph{bounded solution at infinity}, 
whenever the Lyapunov equation, $A\Sigma + \Sigma A\t + D = 0$ has a positive-definite solution. 

From Assumption \ref{assump:point_limit}, only the first case \ref{cond:3}  holds true. 
Consequently, the largest invariant subset of $\{\dot \V = 0 \}$ is the set
\begin{align*}
    \mathcal{I}   \Let \left\{ (\Sigma,\Pi) \in \mathbb{S}_{++}^n \times \{0\}  \;\middle\vert\;  
\begin{array}{@{}l@{}}
A\Sigma + \Sigma A\t + D = 0 \text{ and} \\ ~~~\trace(Q\Sigma) = 0
\end{array}
\right\}.
\end{align*}
It follows from LaSalle's invariance principle~\cite[Theorem~4.4]{ref_khalil}, that all trajectories asymptotically converge to $\mathcal{I}$. 

In summary, every bounded forward limiting trajectory approaches the largest invariant subset of $\{\dot{\V} = 0 \}$. 
Hence, the invariant set is precisely $\mathcal G_\infty\times\{0\}$ and
\begin{align*}
    \Pi^+(t)\to 0, \quad \text{dist}\bigl(\Sigma^+(t),\mathcal G_\infty\bigr)\to 0, \quad t\to +\infty.
\end{align*}
Now, consider the backward limiting arc $(\Sigma^-(\cdot),\Pi^-(\cdot))$, which satisfies the time-reversed covariance--Riccati dynamics. 
For each time horizon $\tf^k$, define,
\begin{align*}
    \tau \Let \tf^k-t, \quad \tau\in[0,\tf^k],
\end{align*}
so that $\tau=0$ corresponds to the terminal endpoint. 
The associated reverse trajectories are then given by \eqref{eqn:time_rev_cov_ricc}. The infinite-horizon backward arc $(\Sigma^-(\cdot),\Pi^-(\cdot))$ is then obtained by passing to the limit, along a subsequence, on compact intervals in $\tau$.
Moreover, define the reversed accumulator $\hat\A(\tau)\Let\A(\tf^k-\tau)$, 
$\mathring{\hat{\A}}(\tau)=-\trace(\hat\Pi(\tau)D)$, 
where $(\,\mathring{}\,)$ denotes differentiation with respect to $\tau$,
and let $\W(\tau)\Let-\trace(\hat\Pi(\tau) \hat\Sigma(\tau))+\hat\A(\tau)$. 
Since $\W$ is non-increasing and bounded, from \eqref{eqn:time_rev_cov_ricc}, 
\begin{align*}
    \mathring{\W}(\tau) =-\trace(Q\Sigma(\tau))-\trace(\Pi(\tau)\t B R^{-1} B\t \Pi(\tau) \Sigma(\tau))\le 0. 
\end{align*}
Once again, LaSalle's invariance principle applies to the reversed-time dynamics and yields
\begin{align*}
    \Pi^-(\tau)\to 0,
\quad
\text{dist}\bigl(\Sigma^-(\tau),\mathcal G_\infty\bigr)\to 0,
\quad \tau\to+\infty.
\end{align*}
Consequently, both the forward and backward limiting arcs approach the same asymptotic gateway set.

\medskip

\begin{figure*}[!t]
    \centering
    \includegraphics[scale=0.33]{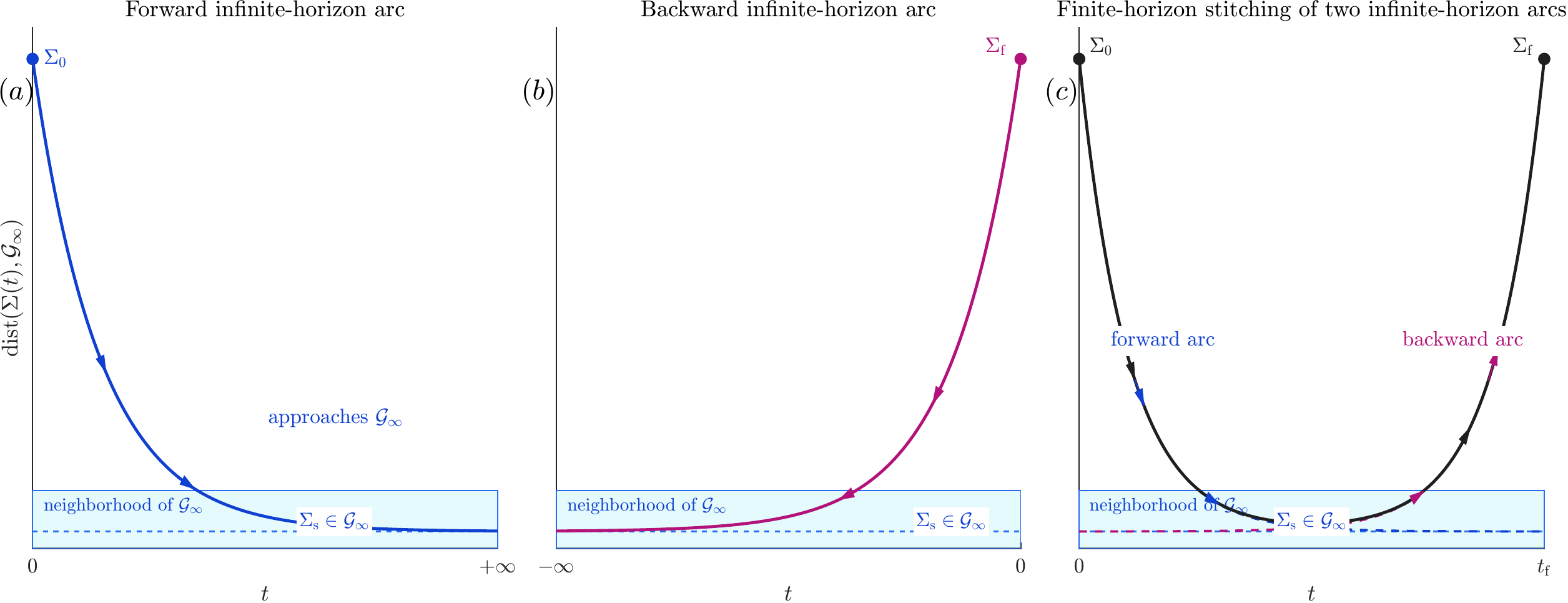}
    \caption{Turnpike interpretation of Asymptotic Connections.
    (a) shows a forward infinite-horizon arc approaching the Asymptotic Gateway $\mathcal{G}_\infty$ from the initial covariance $\initcov $.
    (b) shows a backward infinite-horizon arc approaching $\mathcal{G}_\infty$ in reverse time from the terminal covariance $\fincov $.
    (c) illustrates the corresponding finite-horizon stitching: for sufficiently large $\tf$, the optimal trajectory approaches $\mathcal{G}_\infty$, remains near a stationary covariance $\Sigma_{\rm s}\in\mathcal{G}_\infty$ for an extended interval, and then departs to satisfy the prescribed terminal boundary condition.
    }
    \label{fig:AC_turnpike_stitching}
\end{figure*}

\noindent\proofstep{step:four}{Relation with the asymptotic connection notion}
\ref{step:one} to \ref{step:three} show that the family of optimal covariance/co-state trajectories is %uniformly bounded, equi-Lipschitz, and therefore 
equicontinuous on compact intervals. 
Consequently, every sequence $\tf^k\to+\infty$ admits a subsequence whose associated optimal covariance trajectories converge uniformly on compact intervals to forward and backward limiting arcs approaching $\mathcal G_\infty$. 

Since the covariance/co-state pair also depends continuously on the final time $\tf$ (see Lemma~\ref{lem:reg:init_cost}),
this asymptotic behavior is inherited by the finite-horizon optimal trajectories for all sufficiently large final times. 
Therefore, for every neighborhood of $\mathcal G_\infty$, the corresponding optimal covariance trajectories enter that neighborhood when $\tf$ is sufficiently large. 
In the sense of Definition~\ref{asymp_connection_cov}, this yields an Asymptotic Connection for the family $\{\Sigma_{\tf}(\cdot)\}_{\tf>0}$ in the limit $\tf\to+\infty$.

\medskip

\noindent\proofstep{step:five}{Asymptotic Hamiltonian-zero}
Since $\Pi^\pm(\cdot)\to 0$ and $\text{dist}\bigl(\Sigma^\pm(\cdot),\mathcal G_\infty\bigr)\to 0$, every limit point $\Sigma_{\rm s}$ of the limiting covariance trajectories satisfies,
\begin{align*}
    A\Sigma_{\rm s}+\Sigma_{\rm s}A\t +D=0, \quad \trace(Q\Sigma_{\rm s})=0.
\end{align*}
Therefore, by the time-invariance and the continuity of the Hamiltonian \eqref{final_ham_dyn} with respect to final time $\tf$ from Lemma~\ref{lem:reg:init_cost},
it follows that $\hamil_{\tf}(\Sigma^\pm(\cdot),\Pi^\pm(\cdot))\to 0$ as $\tf \to +\infty$.
Accordingly, under the stated assumptions, the corresponding Hamiltonian converges to zero asymptotically as $\tf \to +\infty$, and $\tf=+\infty$ is an asymptotic Hamiltonian-zero for \Cref{prob:Det:Cov:Steer}. 
\end{proof}

Recall that Theorem~\ref{prop_exist_AC} showed that  $\asmgate \not= \varnothing$ implies an asymptotic Hamiltonian-zero. This lemma shows that an empty gateway rules out this possibility. We will employ this result in the proof of Theorem \ref{thm:cov-steering}.

\begin{myOCP}
\begin{lemma}
\label{lem_empty_gateway}
Consider the family of fixed-final time covariance steering problems associated with \Cref{prob:Det:Cov:Steer}, with $\varpi=0$.
Let Assumption~\ref{assump:standing} and Assumption~\ref{assump:point_limit} hold.
If $\mathcal G_\infty=\varnothing$, then
\begin{align*}
\lim_{\tf\to+\infty} \hamil_{\tf}(\cdot)>0,
\end{align*}
i.e., $\tf=+\infty$ cannot be an asymptotic Hamiltonian-zero.
\end{lemma}
\end{myOCP}

\begin{proof}
It follows from Assumption~\ref{assump:standing:1}, the compactness argument from  Proposition~\ref{cor:cov_costate_compactness}, and the Riccati and covariance dynamics \eqref{final_ricc_dyn} and \eqref{final_cov_dyn}, that the family $\{\dot\Sigma_{\tf}(\cdot)\}_{\tf>0}$ is uniformly continuous on compact subintervals.
From Assumption~\ref{assump:point_limit}, the limiting covariance trajectory $\aset[\big]{(\Sigma_{\tf}(\cdot))}_{\tf>0}$ as $\tf\to +\infty$ is bounded and  admits a well-defined limit as $t\to+\infty$. 
From the uniform continuity of $t \mapsto \dot\Sigma(t)$ on the interval $t\in[0,+\infty)$, Barbalat's Lemma~\cite[Lemma~8.2]{ref_khalil} implies that $\lim_{t\to+\infty}\dot\Sigma(t)=0$ along any such limiting trajectory. 
Specifically, Assumption~\ref{assump:point_limit}, together with $\lim_{t\to+\infty}\dot\Sigma(t)=0$, guarantees that the Hamiltonian admits a well-defined limit as $\tf\to+\infty$, i.e., $\hamil^\infty(\cdot)\Let\lim_{\tf\to+\infty}\hamil_{\tf}(\cdot)$ exists.

The Hamiltonian \eqref{final_ham_dyn} can be rewritten in terms of the covariance dynamics as follows
\begin{align}
\hamil_{\tf}(t)= \trace\Big( \half\dot \Sigma(t) \Pi(t) + \half\Pi(t)\t BR^{-1}B\t\Pi(t)\Sigma(t)  +\half Q\Sigma(t)\Big).
\end{align}
Assume, for contradiction, that $\lim_{\tf\to+\infty} \hamil_{\tf}(t)=0$. Since the dynamics are autonomous, $\hamil_{\tf}(t) =\hamil(\Sigma_{\tf}(t),\Pi_{\tf}(t))$
is constant along each optimal extremal trajectory for every fixed time $t>0$. 
Passing first to the limiting trajectory as $\tf\to+\infty$ and then letting $t\to+\infty$, using $\lim_{t\to+\infty}\dot\Sigma(t)=0$, we obtain,
\begin{align} \label{lemma_v_1_eqn1}
    \lim_{\tf \to +\infty}\hamil_{\tf}(t) = \half\lim_{t \to +\infty} \trace\Big(\Pi(t)\t   BR^{-1}B\t\Pi(t)\Sigma(t) +Q\Sigma(t)\Big),
\end{align}
evaluated along the limiting trajectory, and since both terms inside the trace in \eqref{lemma_v_1_eqn1} are nonnegative, it follows that each term must vanish for the equality $\lim_{\tf\to+\infty} \hamil_{\tf}(t)=0$ to hold. 
Hence, we must have that $\trace(Q\Sigma(t))\to 0$ and $\trace\big(\Pi(t)BR^{-1}B\t \Pi(t)\Sigma(t)\big)\to 0$ as $t \to +\infty$.
This implies that, since $R \succ 0$ and $\Sigma(t)\succ 0$ for all $t\ge 0$, one must have $B\t \Pi(t)\to 0$. 
Finally, using the covariance dynamics together with $\dot \Sigma(t)\to 0$ and $B\t \Pi(t)\to 0$, every limit point $\Sigma_{\rm s}$ of $\Sigma(\cdot)$ satisfies $A \Sigma_{\rm s} + \Sigma_{\rm s} A\t + D=0$ and $\trace(Q \Sigma_{\rm s})=0$. 
Hence, $\Sigma_{\rm s}\in\mathcal G_\infty$, contradicting $\mathcal G_\infty=\varnothing$. Therefore, $\lim_{\tf \to +\infty} \hamil_{\tf}(\cdot) \neq  0$. 
Since $\lim_{\tf \to +\infty} \hamil_{\tf}(\cdot) \ge 0$, it follows that $\lim_{\tf \to +\infty} \hamil_{\tf}(\cdot)>0$.
\end{proof}

As promised in \S \ref{subsect_V_D}, we provide an analysis of the asymptotic limits of the Hamiltonian $\hamil_{\tf}(\cdot)$ as the final time $\tf$ approaches zero and infinity. This result will also be employed in the proof of Theorem \ref{thm:cov-steering} ahead.

\begin{myOCP}
\begin{lemma}[Asymptotic Behavior of the Hamiltonian]
\label{inf_hor_cs}
Assume $Q \succcurlyeq 0$, $R\succ 0$, $\varpi \ge 0$, and let  Assumption~\ref{assump:point_limit} hold. 
For each fixed final time $\tf$, let $\hamil_{\tf}(\cdot)$ denote the associated Hamiltonian of \Cref{prob:Det:Cov:Steer}. 
Then, the following asymptotic limits hold:
\begin{enumerate}[label={\textup{(\ref{inf_hor_cs}-\alph*)}}, leftmargin=*, widest=b, align=left]
\item \label{lemm:cond:1} As $\tf \to 0^+$,
    \begin{align*}
    \lim_{\tf \to 0^+} \hamil_{\tf}(\cdot) = -\infty,
    \end{align*}
    \item \label{lemm:cond:2} As $\tf \to +\infty$,
    \begin{align*}
    \hamil^\infty(\cdot) \Let \lim_{\tf \to +\infty} \hamil_{\tf}(\cdot)\ge \frac{\varpi}{2}.
    \end{align*}
    If, in addition, $\mathcal{G}_\infty=\varnothing$, then $\hamil^\infty(\cdot)>\frac{\varpi}{2}$.
\end{enumerate}
\end{lemma}
\end{myOCP}

\begin{proof}
    The first assertion \ref{lemm:cond:1} follows from the short-time limit behavior of the state transition matrices \eqref{STM_matr}. 
    Specifically, 
    we first show that both $\|B\t\Pi_{\tf}(0)\|\to+ \infty$ and $\|\Pi_{\tf}(0)\|\to+ \infty$ as $\tf\to0^+$.
    To this end,
    assume, by contradiction, that $\sup_{t\in[0,\tf]}\|B\t\Pi_{\tf}(t)\|\le M<+\infty$ along some sequence $\tf\to0^+$.
    Then, the closed-loop matrix $\bar A(t)=A-BR^{-1} B\t\Pi_{\tf}(t)$ satisfies $\sup_{t\in[0,\tf]} \|\bar A(t)\|\le \|A\|+\|BR^{-1}\|M \Let c_M$.
    An application of Gr\"{o}nwall's inequality then yields
    $\|\bar\Phi_{\bar A}(t,0)\| \le e^{c_M t}$ for all $t\in[0,\tf]$.
    Since $\bar\Phi_{\bar A}(\tf,0)-I_n=\int_0^{\tf}\bar A(s)\,\bar\Phi_{\bar A}(s,0)\,\d s$, we therefore obtain
    \begin{align*}
        \|\bar\Phi_{\bar A}(\tf,0)-I_n\| & \le \int_0^{\tf}
        \| \bar A(s) \| \| \bar\Phi_{\bar A}(s,0) \| \, \d s \\
        &\le \int_0^{\tf} c_M\, e^{c_M s}\,\d s =e^{c_M \tf}-1 \rightarrow 0  \quad\text{as } \tf\to0^+ .
    \end{align*}
    Moreover, the diffusion integral term from \eqref{cov_analytic_stm} satisfies
    \begin{align*}
        \lim_{\tf \to0^+}\int_0^{\tf} \bar\Phi_{\bar A}(t,s)\,D\,\bar\Phi_{\bar A}(t,s)\t\, \d s = 0.
    \end{align*}
    Taking the limit $\tf\to0^+$ in \eqref{cov_analytic_stm}, we get
    \begin{align*}
    \fincov = \lim_{\tf\to0^+}\Sigma_{\tf}(\tf) = \initcov ,
    \end{align*}
    contradicting the boundary condition $\fincov \neq\initcov$.
    Therefore, $\sup_{t\in[0,\tf]}\|B\t \Pi_{\tf}(t)\|\to + \infty$ as $\tf\to0^+$.
    Furthermore, since $\|B\t\Pi_{\tf}(0)\| \le \|B\|\,\|\Pi_{\tf}(0)\|$, we also obtain $\|\Pi_{\tf}(0)\|\to+\infty$.

  Next, using Kronecker calculus and lifting the covariance dynamics into a linear control problem with the usual transformation $U\triangleq B\t \Pi \Sigma$, 
    it is not difficult to show that the quadratic term $\frac12 \trace(\Pi_{\tf}(0)BR^{-1}B\t\Pi_{\tf}(0)\initcov )$ in the expression of the Hamiltonian in \eqref{final_ham_dyn}
    grows at least with the rate $\tf^{-2 {c_K}}$ as $\tf\to0^+$
    (see \cite{seidman} for more details, where the time integral of the control energy over the time interval $[0,\tf]$ is shown to blow up with rate $\tf^{-2 {c_K}+1}$), with $c_K$ being the order of the least controllable mode associated with $(A,B)$ (i.e., the controllability index). 
    On the other hand, the monotonicity of the Riccati matrix
    from \cite[Lemma 6]{cite17} gives an upper bound on $\Pi_{\tf}(0)$ in terms of state transition matrices, namely,
    \begin{align*}
         \Pi_{\tf}(0) \prec -\Phi_{12}(\tf,0)^{-1} \Phi_{11}(\tf,0).
    \end{align*}
    As $\tf\to 0^+$, $\Phi_{11}(\tf,0)\to I_n$. 
    From \cite[Lemma III.1]{pavone_inv_gram} it follows that an upper bound for the inverse Gramian $\Phi_{12}(\tf,0)^{-1}$ is $\tf^{-2 {c_K}+1 }$.
    An upper bound on the rate of growth
    of $\|\Pi_{\tf}(0)\|$ as $\tf\to0^+$ 
    can therefore be found by combining the boundedness of $\Phi_{11}(\tf,0)$ with the upper bound $\tf^{-2c_K+1}$ on the inverse Gramian $\Phi_{12}(\tf,0)^{-1}$, so that $\|\Pi_{\tf}(0)\|$ grows at most at the rate $\tf^{-2c_K+1}$. 
    Hence, the quadratic term involving $B\t \Pi_{\tf}(0)$ 
    dominates the remaining linear terms involving $\Pi_{\tf}(0)$ in \eqref{final_ham_dyn}. 
    It follows that $\lim_{\tf\to0^+}\hamil_{\tf}(\cdot)=-\infty$.
 
    The second assertion~\ref{lemm:cond:2} follows from Proposition~\ref{cor:cov_costate_compactness} and  Theorem~\ref{prop_exist_AC}, which establishes the result for the case $\varpi=0$, and from the fact that the time-penalty term does not affect the optimal state or co-state trajectories, but only shifts the Hamiltonian by $\varpi/2$. 
    
    For $\varpi=0$, when $\mathcal{G}_\infty\neq\varnothing$, it follows from Theorem~\ref{prop_exist_AC}
    that $\hamil^{\infty}(\cdot)\ge 0$.
    Similarly, 
   when $\mathcal{G}_\infty=\varnothing$ it follows from Lemma~\ref{lem_empty_gateway} that $\hamil^{\infty}(\cdot)\ge 0$. 
    Since the time penalty  shifts $\hamil_{\tf}(\cdot)$ by $\varpi/2$ without altering the optimal trajectories $(\Sigma,\Pi)$, it follows that $\hamil^{\infty}(\cdot)\ge {\varpi}/{2}$. If, in addition, $\mathcal{G}_\infty=\varnothing$, then Lemma~\ref{lem_empty_gateway} yields the strict inequality.
\end{proof}

\begin{proof}[\highlight{Proof of Theorem \ref{thm:cov-steering}}]
The ``if'' part follows directly from Proposition~\ref{cor:cov_costate_compactness} and Theorem~\ref{prop_exist_AC}. 

For the ``only if'' part, suppose that the condition fails, either because $\varpi >0$ or because $\varpi=0$ and $\mathcal G_\infty = \varnothing$. 
For the first case, we have $\hamil^\infty(\cdot) \ge \varpi/2>0$ from Lemma~\ref{lemm:cond:2}; 
for the latter case, we have $\hamil^\infty(\cdot)>0$ from Lemma~\ref{lem_empty_gateway}. 
In either case, $\lim_{\tf \to \infty} \hamil_{\tf}(\cdot)>0$, and therefore, from Definition~\ref{Def:HamZeros:b}, $\tf=+\infty$ is not an asymptotic Hamiltonian-zero.
\end{proof}

%% file: sections/appendixC.tex
\section{}\label{Appendix_Thm}

\begin{proof}[\highlight{Proof of Theorem~\ref{thm:p1d-main}}]
From Corollary~\ref{cor_dv_H_0}, any finite candidate optimal final time $\tfc$ must satisfy the transversality condition \eqref{total_hamiltonian}, i.e., $\hamil_{\tfc}(\cdot)=0$. Thus, it suffices to show that the map $\tf \mapsto \hamil_{\tf}(\cdot)$ has at least one finite zero.
It is shown in Lemma \ref{inf_hor_cs} (see Appendix~\ref{sec:appen:b}) that, for \Cref{prob:Det:Cov:Steer}, 
one has $\lim_{\tf\to 0^+}\hamil_{\tf}(\cdot)=-\infty$.
Moreover, under condition \textit{a)}, Lemma \ref{inf_hor_cs} asserts that $\lim_{\tf\to+\infty}\hamil_{\tf}(\cdot)\ge {\varpi}/{2}$ and thus $\lim_{\tf\to+\infty}\hamil_{\tf}(\cdot)>0$. 
Also, under any one of the conditions \textit{b)}--\textit{c)}, Lemma \ref{inf_hor_cs} asserts that
the inequality is strict, i.e., $\lim_{\tf\to+\infty}\hamil_{\tf}(\cdot)>{\varpi}/{2}$, therefore $\lim_{\tf\to+\infty}\hamil_{\tf}(\cdot)>0$. 
Since $\hamil_{\tf}(\cdot)$ is continuous with respect to $\tf$ (see Lemma \ref{lem:reg:init_cost}),
the intermediate value theorem implies the existence of at least one $\tfc\in(0,+\infty)$ such that $\hamil_{\tfc}(\cdot)=0$.
Hence, \Cref{prob:Det:Cov:Steer} admits at least one finite candidate free-final time minimizer. Our proof is complete.
\end{proof}